\documentclass[reqno]{amsart}
\usepackage[utf8]{inputenc}
\usepackage[top=1.4in, bottom=1.2in, left=1.2in, right=1.2in]{geometry}

\usepackage{amssymb,amsfonts,amsmath,amsthm,tikz,tikz-cd,graphicx,url,mathtools,thmtools}
\usepackage[mode=buildnew]{standalone}
\usepackage{xcolor}
\usepackage[all,arc]{xy}
\usepackage{enumerate}
\usepackage[colorlinks=true]{hyperref}
\usepackage[alphabetic]{amsrefs}
\usepackage{comment}
\usepackage{subfigure}
\usetikzlibrary{matrix}
\usepackage{cmap}
\usepackage{xparse}
\usepackage{stmaryrd}

 \usepackage[sc,osf]{mathpazo}
 \usepackage{upgreek}
 \usepackage{optparams,eurosym}

\newtheorem{thm}{Theorem}[section]

\newtheorem{prop}[thm]{Proposition}
\newtheorem{lem}[thm]{Lemma}

\theoremstyle{definition}

\newtheorem{exmp}[thm]{Example}

\theoremstyle{remark}
\newtheorem{rem}[thm]{Remark}

\newcommand{\Z}{\mathbb Z}
\newcommand{\RP}{\mathbb{RP}}
\newcommand{\ZZ}{\mathbb{Z}/2}
\newcommand{\QQ}{\mathbb{Q}}
\newcommand{\RR}{\mathbb{R}}
\newcommand{\FF}{\mathbb{F}}

\DeclareMathOperator{\idd}{id}

\DeclareDocumentCommand{\BNP}{o}{%
  \mathbf{BN}_{\mathbb{RP}^2\IfValueT{#1}{,#1}}%
}
\DeclareDocumentCommand{\LinkP}{o}{%
  \mathbf{Links}_{\mathbb{RP}^3\IfValueT{#1}{,#1}}^\circ%
}
\DeclareDocumentCommand{\LinkIP}{o}{%
  \mathbf{Links}_{I\tilde\times\mathbb{RP}^2\IfValueT{#1}{,#1}}^\circ%
}
\DeclareDocumentCommand{\LinkUP}{o}{%
  \mathbf{Links}_{u\mathbb{RP}^3\IfValueT{#1}{,#1}}%
}
\newcommand{\Kbpm}{\mathbf K_{/\pm}^b}

\DeclareMathOperator{\Diff}{Diff^+}
\DeclareMathOperator{\Diffb}{Diff_*^+}

\DeclareMathOperator{\Kh}{Kh}

\newcommand{\lrb}[1]{\llbracket #1 \rrbracket}

\DeclareMathOperator{\SO}{SO}
\DeclareMathOperator{\sw}{sw}

\DeclareMathOperator{\fl}{fl}

\title{Intrinsic Khovanov homology in $\RP^3$}

\author{Qiuyu Ren}
\address{Department of Mathematics, University of California, Berkeley, Berkeley, CA 94720, USA}
\email{qiuyu\_ren@berkeley.edu}

\author{Hongjian Yang}
\address{Department of Mathematics, Stanford University, Stanford, CA 94305, USA}
\email{yhj@stanford.edu}

\tikzset{anchorbase/.style={baseline={([yshift=-0.5ex]current bounding box.center)}}}
\usetikzlibrary{calc}
\usetikzlibrary{decorations.markings}
\usetikzlibrary{decorations.pathreplacing}
\usetikzlibrary{arrows,shapes,positioning}
\tikzstyle directed=[postaction={decorate,decoration={markings,
    mark=at position #1 with {\arrow{>}}}}]
\tikzstyle rdirected=[postaction={decorate,decoration={markings,
    mark=at position #1 with {\arrow{<}}}}]

\newcommand{\classimiddlemove}{
\begin{tikzpicture}[anchorbase,scale=.7]

    \draw[thick,->] (-8,0)--(-7,0);
    \node[font=\small] at (-7.5,0.5) {\text{RII}};
    \draw[thick,->] (-3,0) to (-2,0);
    \node[font=\small] at (-2.5,0.5) {\text{isotopy}};
    \draw[thick,->] (2,0) to (3,0);
    \node[font=\small] at (2.5,0.5) {\text{RIII}};
    \draw[thick,->] (7,0) to (8,0);
    \node[font=\small] at (7.5,0.5) {\text{RII}};

    \node at (-10,0) {
    \begin{tikzpicture}[scale=.3]
        \draw (-2,0) rectangle (2,5);
        \draw [very thick] (-1,0) to (-1,5);
        \draw [very thick] (0,0) to (0,5);
        \draw [very thick] (1,0) to (1,5);
    \end{tikzpicture}
    };
    
\node at (-5,0) {
    \begin{tikzpicture}[anchorbase,scale=.3]
        \draw (-1,0) rectangle (3,5);
        \draw [very thick] (0,0) to [out=90,in=270] (2,2.5) to [out=90,in=270] (0,5);
        \draw [white,line width=0.15cm] (1,0) to [out=90,in=270](2,1.25)to [out=90,in=270]  (1,2.5) to[out=90,in=270] (2,3.75) to[out=90,in=270] (1,5);
        \draw [very thick] (1,0) to [out=90,in=270](2,1.25)to [out=90,in=270]  (1,2.5) to[out=90,in=270] (2,3.75) to[out=90,in=270] (1,5);
        \draw [white,line width=0.15cm] (2,0) to [out=90,in=270](0,2.5)to [out=90,in=270]  (2,5);
        \draw [very thick] (2,0) to [out=90,in=270](0,2.5)to [out=90,in=270]  (2,5);
    \end{tikzpicture}
};

\node at (0,0) {
    \begin{tikzpicture}[anchorbase,scale=.3]
        \draw (-1,0) rectangle (3,5);
        \draw [very thick] (2,0) to [out=90,in=270] (0,2.5) to [out=90,in=270] (2,5);
    
        \draw [white,line width=0.15cm] (1,0) to [out=90,in=270](2,1.25)to [out=90,in=270] (0,3.75) to[out=90,in=270] (1,5);
        \draw [very thick] (1,0) to [out=90,in=270](2,1.25)to [out=90,in=270] (0,3.75) to[out=90,in=270] (1,5);
        
        \draw [white,line width=0.15cm] (0,0) to [out=90,in=270](2,2.5)to [out=90,in=270]  (0,5);
        \draw [very thick] (0,0) to [out=90,in=270](2,2.5)to [out=90,in=270]  (0,5);
    \end{tikzpicture}
};

\node at (5,0) {
    \begin{tikzpicture}[anchorbase,scale=.3]
        \draw (-1,0) rectangle (3,5);
        \draw [very thick] (2,0) to [out=90,in=270] (0,2.5) to [out=90,in=270] (2,5);
        
        \draw [white,line width=0.15cm] (1,0) to [out=90,in=270](0,1.25)to [out=90,in=270] (1,2.5) to [out=90,in=270] (0,3.75) to[out=90,in=270] (1,5);
        \draw [very thick] (1,0) to [out=90,in=270](0,1.25)to [out=90,in=270] (1,2.5) to [out=90,in=270] (0,3.75) to[out=90,in=270] (1,5);
        
        \draw [white,line width=0.15cm] (0,0) to [out=90,in=270](2,2.5)to [out=90,in=270]  (0,5);
        \draw [very thick] (0,0) to [out=90,in=270](2,2.5)to [out=90,in=270]  (0,5);
    \end{tikzpicture}
};

\node at (10,0) {
    \begin{tikzpicture}[scale=.3]
        \draw (-2,0) rectangle (2,5);
        \draw [very thick] (-1,0) to (-1,5);
        \draw [very thick] (0,0) to (0,5);
        \draw [very thick] (1,0) to (1,5);
    \end{tikzpicture}
};
\end{tikzpicture}
}

\newcommand{\movieconst}{
    \coordinate (O) at (0,0); 
    \coordinate (T) at (0,4); 
    \draw[thick] ($(O) + (-2,0)$) -- ($(T) + (-2,0)$); 
    \draw[thick] ($(O) + (2,0)$) -- ($(T) + (2,0)$);
    
    \draw[dashed, thick] ($(O) + (2,0)$) arc (0:180:2 and 0.6) coordinate[pos=0.25] (P1) coordinate[pos=0.75] (P4);
    \draw[thick] ($(O) - (2,0)$) arc (180:360:2 and 0.6)  coordinate[pos=0.25] (P3) coordinate[pos=0.75] (P2);

    \draw[thick] ($(T) + (2,0)$) arc (0:180:2 and 0.6) coordinate[pos=0.25] (P5) coordinate[pos=0.75] (P8);
    \draw[thick] ($(T) - (2,0)$) arc (180:360:2 and 0.6)  coordinate[pos=0.25] (P7) coordinate[pos=0.75] (P6);

    \node at (0,2) {
        \begin{tikzpicture}[anchorbase,scale=0.15]
            \draw[thick] (0,0)--(4,0)--(6,2)--(2,2)--(0,0);
            \draw[thick] (0.5,0)--(-0.5,-1) (3.5,0)--(2.5,-1)
                        (2.5,2)--(3.5,3) (5.5,2)--(6.5,3);
            \node at (1.5,-0.5) {\text{...}};
            \node at (4.5,2.5) {\text{...}};
            \node[font=\tiny,xslant=0.75] at (3,1) {$\textbf{T}$};
        \end{tikzpicture}
    };
}

\newcommand{\moveA}[1]{
    \begin{tikzpicture}[anchorbase,scale=#1]
        \movieconst
        \draw[->,very thick,blue] ($(T)+(-1.3,0)$) arc (90:270:0.3 and 0.1);
        \draw[very thick,blue] ($(T)+(-1.3,-0.2)$) arc (-90:90:0.3 and 0.1);
    \end{tikzpicture}
}

\newcommand{\moveB}[1]{
    \begin{tikzpicture}[anchorbase,scale=#1]
        \movieconst
        \draw[very thick,blue] (P1) to[out=240,in=90] ($(O)+(1.3,0)$);
        \draw[very thick,<-,blue] ($(O)+(1.3,0)$) to[out=270,in=120] (P2);
        \draw[very thick,-<,blue] (P7) to[out=60,in=90] ($(T)-(1.3,0)$);
        \draw[very thick,blue] (P8) to[out=-60,in=90] ($(T)-(1.3,0)$);
    \end{tikzpicture}
}

\newcommand{\moveC}[1]{
    \begin{tikzpicture}[anchorbase,scale=#1]
        \movieconst
        \draw[very thick,blue] (P5) to[out=240,in=90] ($(T)+(1.3,0)$);
        \draw[very thick,<-,blue] ($(T)+(1.3,0)$) to[out=270,in=120] (P6);
        \draw[very thick,-<,blue] (P3) to[out=60,in=90] ($(O)-(1.3,0)$);
        \draw[very thick,blue] (P4) to[out=-60,in=90] ($(O)-(1.3,0)$);
    \end{tikzpicture}
}

\newcommand{\moveD}[1]{
    \begin{tikzpicture}[anchorbase,scale=#1]
        \movieconst
        \draw[very thick,blue] ($(T)+(1.3,0.2)$) arc (90:270:0.3 and 0.1);
        \draw[very thick,<-,blue] ($(T)+(1.3,0)$) arc (-90:90:0.3 and 0.1);
    \end{tikzpicture}
}

\newcommand{\moveE}[1]{
     \begin{tikzpicture}[anchorbase,scale=#1]
        \movieconst
        \draw[very thick,blue] ($(T)-(1.3,0)$) arc (90:270:0.3 and 0.1);
        \draw[very thick,<-,blue] ($(T)-(1.3,0.2)$) arc (-90:90:0.3 and 0.1);
    \end{tikzpicture}
}

\newcommand{\diagmovconst}{
    \draw[rounded corners] (0,0) rectangle (6,6);

    \draw[very thick] (2.2,2.6) rectangle (3.8,3.4);
    \node[font=\tiny] at (3,3) {\text{$T$}};
    \node at (3,2) {\text{...}};
    \node at (3,4) {\text{...}};
    \draw[white,line width=0.15cm] (2.4,0.1)--(2.4,2.5) (2.4,3.5)--(2.4,5.9) (3.6,0.1)--(3.6,2.5) (3.6,3.5)--(3.6,5.9);
    \draw[very thick] (2.4,0)--(2.4,2.6) (2.4,3.4)--(2.4,6) (3.6,0)--(3.6,2.6) (3.6,3.4)--(3.6,6);   
}

\newcommand{\diagmovA}[1]{
\begin{tikzpicture}[anchorbase,scale=#1]
    \diagmovconst
    \draw[blue,very thick] (1,3) circle (0.3);
\end{tikzpicture}
}

\newcommand{\diagmovB}[1]{
\begin{tikzpicture}[anchorbase,scale=#1]
    \diagmovconst
    \draw[blue,very thick] (1,0) to[out=75,in=-75] (1,6);
    \draw[blue,very thick] (5,0) to[out=105,in=-105] (5,6);
\end{tikzpicture}
}

\newcommand{\diagmovC}[1]{
\begin{tikzpicture}[anchorbase,scale=#1]
    \draw[blue,very thick] (1,0)--(1,0.5) to[out=90,in=-90] (5,3) to[out=90,in=-90] (1,5.5)--(1,6);
    \diagmovconst
    \draw[white, line width=0.15cm] (5,0) to[out=90,in=-90] (1,2)--(1,4) to[out=90,in=-90] (5,6);
    \draw[blue,very thick] (5,0) to[out=90,in=-90] (1,2)--(1,4) to[out=90,in=-90] (5,6);
\end{tikzpicture}
}

\newcommand{\diagmovD}[1]{
\begin{tikzpicture}[anchorbase,scale=#1]
    \draw[rounded corners] (0,0) rectangle (6,6);
    
    \draw[blue,very thick] (5,6)--(5,5.2) to[out=-90,in=90] (1,2.6) to[out=-90,in=90] (5,0.3)--(5,0);
    
    \draw[very thick] (2.2,4.6) rectangle (3.8,5.4);
    \node[font=\tiny] at (3,5) {\text{$T$}};
    \node at (3,3) {\text{...}};
    \node at (3,5.7) {\text{...}};
    \draw[white,line width=0.15cm] (2.4,0.1)--(2.4,4.5) (2.4,5.5)--(2.4,5.9) (3.6,0.1)--(3.6,4.5) (3.6,5.5)--(3.6,5.9);
    \draw[very thick] (2.4,0)--(2.4,4.6) (2.4,5.4)--(2.4,6) (3.6,0)--(3.6,4.6) (3.6,5.4)--(3.6,6);  

    \draw[white, line width=0.15cm] (1,5.9)--(1,4.5) to[out=-90,in=90] (5,1.8) to[out=-90,in=90] (1,0.1);
    \draw[blue,very thick] (1,6)--(1,4.5) to[out=-90,in=90] (5,1.8) to[out=-90,in=90] (1,0);

\end{tikzpicture}
}

\newcommand{\diagmovE}[1]{
\begin{tikzpicture}[anchorbase,scale=#1]
    \diagmovconst
    \draw[blue,very thick] (5,3) circle (0.3);
\end{tikzpicture}
}

\newcommand{\diagmovF}[1]{
\begin{tikzpicture}[anchorbase,scale=#1]
    \diagmovconst
    \draw[white, line width=0.15cm] (3,3) circle (2.2);
    \draw[blue,very thick] (3,3) circle (2.2);
\end{tikzpicture}
}

\newcommand{\twoelp}{
    \coordinate (O) at (0,0); 
    \coordinate (T) at (0,3); 
    \draw[dashed, thick] ($(O) + (1,0)$) arc (0:180:1 and 0.4);
    \draw[thick] ($(O) - (1,0)$) arc (180:360:1 and 0.4);
    \draw[thick] ($(T) + (1,0)$) arc (0:180:1 and 0.4);
    \draw[thick] ($(T) - (1,0)$) arc (180:360:1 and 0.4) ;
}

\newcommand{\typeimonodromy}{
\begin{tikzpicture}[anchorbase,scale=.7]

    \draw[thick,->] (-3,0) to (-2,0);
    \draw[thick,->] (2,0) to (3,0);
    \draw[thick,->] (5,0) to (6,0);

    \node at (-5,0) {
    \begin{tikzpicture}[scale=.4]
        \draw[rounded corners] (-2.5,0) rectangle (2.5,5);
        \draw[very thick] (-0.7,0)--(-0.7,5) (0.7,0)--(0.7,5);
        \fill[white] (-1,2) rectangle (1,3);
        \draw[very thick](-1,2) rectangle (1,3);
        \node[font=\tiny] at (0,2.5) {\text{$T$}};
        \node at (0,1) {\text{...}};
        \node at (0,4) {\text{...}};
    \end{tikzpicture}
    };
    
\node[rotate=90] at (0,0) {
    \begin{tikzpicture}[anchorbase,scale=.4]
        \draw[rounded corners] (-2.5,0) rectangle (2.5,5);
        \draw[very thick] (-0.7,0)--(-0.7,5) (0.7,0)--(0.7,5);
        \fill[white] (-1,2) rectangle (1,3);
        \draw[very thick](-1,2) rectangle (1,3);
        \node[font=\tiny] at (0,2.5) {\text{$T$}};
        \node at (0,1) {\text{...}};
        \node at (0,4) {\text{...}};   
    \end{tikzpicture}
};

\node at (4,0) {\text{...}};

\node at (8,0) {
    \begin{tikzpicture}[scale=.4]
        \draw[rounded corners] (-2.5,0) rectangle (2.5,5);
        \draw[very thick] (-0.7,0)--(-0.7,5) (0.7,0)--(0.7,5);
        \fill[white] (-1,2) rectangle (1,3);
        \draw[very thick](-1,2) rectangle (1,3);
        \node[font=\tiny] at (0,2.5) {\text{$T$}};
        \node at (0,1) {\text{...}};
        \node at (0,4) {\text{...}};
    \end{tikzpicture}
};

\end{tikzpicture}
}

\newcommand{\typeiimonodromy}{
\begin{tikzpicture}[anchorbase,scale=.7]

    \draw[thick,->] (-3.5,-4)--(-2.5,-4);
    \node[font=\small] at (-9,-3.5) {\text{isotopy}};
    \draw[thick,->] (-3.5,0) to (-2.5,0);
    \node[font=\small] at (-3,0.5) {\text{RII \& RIII}};
    \draw[thick,->] (2.5,0) to (3.5,0);
    \node[font=\small] at (3,0.5) {\text{RIII}};
    \draw[thick,->] (2.5,-4) to (3.5,-4);
    \node[font=\small] at (-3,-3.5) {\text{isotopy}};
    \node[font=\small] at (3,-3.5) {\text{isotopy}};
    \draw[thick,->] (-9.5,-4) to (-8.5,-4);

    \node at (-6,0) {
    \begin{tikzpicture}[scale=.4]
        \draw[rounded corners] (-2.5,0) rectangle (2.5,5);
        \draw[very thick] (-0.7,0)--(-0.7,5) (0.7,0)--(0.7,5);
        \fill[white] (-1,2) rectangle (1,3);
        \draw[very thick](-1,2) rectangle (1,3);
        \node[font=\tiny] at (0,2.5) {\text{$\beta$}};
        \node at (0,1) {\text{...}};
        \node at (0,4) {\text{...}};
    \end{tikzpicture}
    };

    \node at (0,0) {
    \begin{tikzpicture}[scale=.4]
        \draw[rounded corners] (-2.5,0) rectangle (2.5,5);
        \draw[very thick] (-0.7,0)--(-0.7,5) (0.7,0)--(0.7,5);
        \fill[white] (-1,2) rectangle (1,3);
        \draw[very thick](-1,2) rectangle (1,3);
        \node[font=\tiny] at (0,2.5) {\text{$\beta$}};
        \fill[white] (-1,3.2) rectangle (1,4.2);
        \draw[very thick](-1,3.2) rectangle (1,4.2);
        \node[font=\tiny] at (0,3.7) {\text{$\Delta^{-1}$}};
        \fill[white] (-1,0.8) rectangle (1,1.8);
        \draw[very thick](-1,0.8) rectangle (1,1.8);
        \node[font=\tiny] at (0,1.3) {\text{$\Delta$}};
        \node at (0,0.4) {\text{...}};
        \node at (0,4.6) {\text{...}};
    \end{tikzpicture}
};

\node at (6,0) {
    \begin{tikzpicture}[anchorbase,scale=.4]
        \draw[rounded corners] (-2.5,0) rectangle (2.5,5);
        \draw[very thick] (-0.7,0)--(-0.7,5) (0.7,0)--(0.7,5);
        \fill[white] (-1,2) rectangle (1,3);
        \draw[very thick](-1,2) rectangle (1,3);
        \node[font=\tiny] at (0,2.5) {\text{$\beta^*$}};
        \node at (0,1) {\text{...}};
        \node at (0,4) {\text{...}};
    \end{tikzpicture}
};

\node at (-6,-4) {
    \begin{tikzpicture}[anchorbase,scale=.4]
        \draw[rounded corners] (-2.5,0) rectangle (2.5,5);
        \draw[very thick] (-0.7,0)--(-0.7,5) (0.7,0)--(0.7,5);
        \fill[white] (-1,1) rectangle (1,2);
        \draw[very thick](-1,1) rectangle (1,2);
        \node[font=\tiny] at (0,1.5) {\text{$\beta^*$}};
        \node at (0,3.5) {\text{...}};
        \node at (0,0.5) {\text{...}};
    \end{tikzpicture}
};

\node at (0,-4) {
    \begin{tikzpicture}[anchorbase,scale=.4]
        \draw[rounded corners] (-2.5,0) rectangle (2.5,5);
        \draw[very thick] (-0.7,0)--(-0.7,5) (0.7,0)--(0.7,5);
        \fill[white] (-1,3.5) rectangle (1,4.5);
        \draw[very thick](-1,3.5) rectangle (1,4.5);
        \node[font=\tiny] at (0,4) {\text{$\beta$}};
        \node at (0,1.75) {\text{...}};
        \node at (0,4.75) {\text{...}};
    \end{tikzpicture}
};

\node at (6,-4) {
    \begin{tikzpicture}[scale=.4]
        \draw[rounded corners] (-2.5,0) rectangle (2.5,5);
        \draw[very thick] (-0.7,0)--(-0.7,5) (0.7,0)--(0.7,5);
        \fill[white] (-1,2) rectangle (1,3);
        \draw[very thick](-1,2) rectangle (1,3);
        \node[font=\tiny] at (0,2.5) {\text{$\beta$}};
        \node at (0,1) {\text{...}};
        \node at (0,4) {\text{...}};
    \end{tikzpicture}
};
\end{tikzpicture}
}

\newcommand{\classomiddlemove}{
\begin{tikzpicture}[anchorbase,scale=.7]

    \draw[thick,->] (-8,0)--(-7,0);
    \node[font=\small] at (-7.5,0.5) {\text{RII}};
    \draw[thick,->] (-3,0) to (-2,0);
    \node[font=\small] at (-2.5,0.5) {\text{isotopy}};
    \draw[thick,->] (2,0) to (3,0);
    \node[font=\small] at (2.5,0.5) {\text{RII}};

    \node at (-10,0) {
    \begin{tikzpicture}[scale=.3]
        \draw (-2,0) rectangle (2,5);
        \draw [very thick] (-0.5,0) to (-0.5,5);
        \draw [very thick] (0.5,0) to (0.5,5);
    \end{tikzpicture}
    };
    
\node at (-5,0) {
    \begin{tikzpicture}[anchorbase,scale=.3]
        \draw (-1,0) rectangle (3,5);
        \draw [very thick] (0.5,0) to [out=90,in=270] (1.5,2.5) to [out=90,in=270] (0.5,5);
        \draw [white,line width=0.15cm] (1.5,0) to [out=90,in=270](0.5,2.5)to [out=90,in=270]  (1.5,5);
        \draw [very thick] (1.5,0) to [out=90,in=270](0.5,2.5)to [out=90,in=270]  (1.5,5);
    \end{tikzpicture}
};

\node at (0,0) {
    \begin{tikzpicture}[anchorbase,scale=.3]
        \draw (-1,0) rectangle (3,5);
        \draw [very thick] (1.5,0) to [out=90,in=270] (0.5,2.5) to [out=90,in=270] (1.5,5);
        \draw [white,line width=0.15cm] (0.5,0) to [out=90,in=270](1.5,2.5)to [out=90,in=270]  (0.5,5);
        \draw [very thick] (0.5,0) to [out=90,in=270](1.5,2.5)to [out=90,in=270]  (0.5,5);
    \end{tikzpicture}
};

\node at (5,0) {
    \begin{tikzpicture}[scale=.3]
        \draw (-2,0) rectangle (2,5);
        \draw [very thick] (-0.5,0) to (-0.5,5);
        \draw [very thick] (0.5,0) to (0.5,5);
    \end{tikzpicture}
};
\end{tikzpicture}
}

\newcommand{\classitypeiimove}{
\begin{tikzpicture}[anchorbase,scale=.7]

    \draw[thick,->] (-8,0)--(-7,0);
    \node[font=\small] at (-7.5,0.5) {\text{flip}};
    \draw[thick,->] (-3,0) to (-2,0);
    \node[font=\small] at (-2.5,0.5) {\text{RII}};
    \draw[thick,->] (2,0) to (3,0);
    \node[font=\small] at (2.5,0.5) {\text{isotopy}};

    \node at (-10,0) {
    \begin{tikzpicture}[scale=.4]
        \draw[rounded corners] (-2.5,0) rectangle (2.5,5);
        \draw[very thick] (0,0)--(0,5);
        \draw[very thick] (-1,2.5) circle (0.5);
    \end{tikzpicture}
    };
    
\node at (-5,0) {
    \begin{tikzpicture}[scale=.4]
        \draw[rounded corners] (-2.5,0) rectangle (2.5,5);
        \draw[very thick] (0,0)--(0,5);
        \draw[very thick] (-1,2.5) circle (0.5);
    \end{tikzpicture}
};

\node at (0,0) {
    \begin{tikzpicture}[scale=.4]
        \draw[rounded corners] (-2.5,0) rectangle (2.5,5);
        \draw[very thick] (0,0)--(0,5);
        \draw[very thick] (1,2.5) circle (0.5);
    \end{tikzpicture}
};

\node at (5,0) {
    \begin{tikzpicture}[scale=.4]
        \draw[rounded corners] (-2.5,0) rectangle (2.5,5);
        \draw[very thick] (0,0)--(0,5);
        \draw[very thick] (-1,2.5) circle (0.5);
    \end{tikzpicture}
};

\end{tikzpicture}
}

\begin{document}

\begin{abstract}
We prove that Khovanov homology is an invariant of links in unparametrized $\RP^3$'s, i.e., oriented $3$-manifolds diffeomorphic to $\RP^3$. Along the way, we establish the functoriality of Khovanov homology for link cobordisms in $I\times\RP^3$.
\end{abstract}

\maketitle

\setlength{\parskip}{0.25\baselineskip}

\section{Introduction}

People's understanding of functoriality of Khovanov homology \cite{khovanov2000categorification} has gradually deepened over time \cites{jacobsson2004invariant,khovanov2002functor,bar2005khovanov,blanchet2010oriented}. However, the aforementioned literature only considers Khovanov homology as an invariant of links in a \textit{concrete}, or \textit{parametrized}, $\RR^3$ and proves the functoriality for link cobordisms in $I\times \RR^3$. To define $\mathfrak{gl}_N$ skein lasagna modules, Morrison, Walker, and Wedrich \cite{morrison2022invariants} improve the functoriality statement in the following two aspects. First, they prove the functoriality of Khovanov homology (in fact, the $\mathfrak{gl}_N$ link homology \cite{khovanov2008matrix}) for link cobordisms in $I\times S^3$ (rather than $I\times\RR^3$). Second, they prove that Khovanov homology is a functorial invariant of links in an \textit{unparametrized} (oriented) $S^3$, i.e., an oriented $3$-manifold diffeomorphic to $S^3$ via some unspecified diffeomorphism to $S^3$.

In a different direction, a major open problem is to generalize Khovanov homology to links in other $3$-manifolds. In the case of $\RP^3$---arguably the next simplest closed $3$-manifold---Khovanov-type homology has been defined and widely studied \cites{asaeda2004categorification,Gabrovek2013THECO,manolescu2025rasmussen,chen2021khovanov,chen2025bar,yang2025instantons}. As in $S^3$, to define Khovanov homology in $\RP^3$, one removes a point and uses the twisted $I$-bundle structure on $\RP^3\backslash\{*\}$ to make sense of link diagrams (on $\RP^2$). Then, the construction of Asaeda--Przytycki--Sikora \cite{asaeda2004categorification} specializes to define Khovanov homology for links in $\RP^3$, with the caveat that for null-homologous links the theory is only defined over $\FF_2$ instead of $\Z$. We note that Bar-Natan's canopolis formalism \cite{bar2005khovanov} works in the setting of \cite{asaeda2004categorification} for oriented $I$-bundles over surfaces with mild modifications, implying the functoriality of Khovanov homology for link cobordisms in $I\times (\RP^3\backslash\{*\})$, where one takes $\FF_2$ coefficients for cobordisms between null-homologous links. However, the functoriality in $I\times\RP^3$ was previously unknown.

This paper concerns the question of defining Khovanov homology for links in unparametrized $\RP^3$'s, \textit{\`a la} Morrison--Walker--Wedrich \cite{morrison2022invariants}. We state our main theorem in an informal form as follows.

\begin{thm}\label{thm:main thm}
    Khovanov homology in $\RP^3$, with $\FF_2$ coefficients for null-homologous links and $\Z$ coefficients for homologically essential links, upgrades to an invariant of links in unparametrized $\RP^3$'s.
\end{thm}

Along the way to proving Theorem \ref{thm:main thm}, we also establish the functoriality of Khovanov homology for link cobordisms in $I\times\RP^3$.

\begin{thm}\label{thm:functoriality}
    Khovanov homology in $\RP^3$, with $\FF_2$ coefficients for null-homologous links and $\Z$ coefficients for homologically essential links, is functorial (up to overall sign for homologically essential links) for link cobordisms in $I\times \RP^3$.
\end{thm}

In fact, we will prove analogues of Theorem~\ref{thm:main thm} and Theorem~\ref{thm:functoriality} at the level of appropriate Bar-Natan categories, which recover the theorems stated as special cases. The more precise and general theorem statements can be found in Theorem~\ref{thm:unpara_Kh_1} and Theorem~\ref{thm:unpara_Kh_0}, Theorem~\ref{thm:functoriality_1} and Theorem~\ref{thm:functoriality_0}, respectively. Each general statement is split into two theorems because link homology theories in $\RP^3$ exhibit different behaviors for null-homologous links and homologically essential links, which require different treatments at the level of Bar-Natan categories and can recover different link homology theories in the literature; see Remark~\ref{rem:concrete_unpara} and Remark~\ref{rem:concrete_functorality}.

\begin{rem}
    An integral lift of Khovanov homology for null-homologous links in $\RP^3$ is obtained by Gabrov\v sek \cite{Gabrovek2013THECO} and Manolescu--Willis \cite{manolescu2025rasmussen}. We are currently not able to recover the functoriality of Khovanov homology over $\Z$ for link cobordisms between null-homologous links.
\end{rem}

The proof of Theorem \ref{thm:functoriality} follows a similar strategy to \cite{morrison2022invariants}. The only additional movie move one needs to check is a \textit{sweep-around move} on $\RP^2$; see Figures~\ref{fig:sweeparound} and \ref{fig:sweeparound_diagram}. 
The proof of Theorem \ref{thm:main thm}, on the other hand, is more difficult than the $S^3$ case since $\RP^3$ has a more complicated diffeomorphism group. The generalized Smale Conjecture \cite{bamler2019ricci} implies that the oriented diffeomorphism group of $\RP^3$ has the homotopy type of $\operatorname{SO}(4)/\{\pm1\}$. In particular, we have \[\pi_0(\operatorname{Diff}^+(\RP^3))=1,\quad\pi_1(\operatorname{Diff}^+(\RP^3))\cong\Z/2\oplus\Z/2.\]
There are then two types of movies to check to upgrade Khovanov homology to an invariant of links in unparametrized $\RP^3$'s. The one coming from $\pi_1(\operatorname{SO}(4))$ is essentially the same as in the $S^3$ case. The other one, which comes from the deck transformation, is new and related to the flip map in $S^3$ studied in \cite{chen2025flip} by Chen and the second author. 

The key idea behind Theorems \ref{thm:main thm} and \ref{thm:functoriality} is a filtration argument that reduces computations to crossingless diagrams. Roughly speaking, if a link cobordism from a link diagram $D$ to itself has Euler characteristic zero and has a nice movie representation, it is often feasible to perturb it such that the induced map on Khovanov homology preserves the vertices in the cube of resolutions of $D$. Similar arguments have appeared in \cites{morrison2022invariants,chen2025flip}. 

\subsection*{Organization of the paper}In Section \ref{sec:pre}, we provide necessary background on knot theory in $\RP^3$ and Khovanov homology in $\RP^3$. We then state the explicit forms and prove Theorem \ref{thm:functoriality} in Section \ref{sec:sw} and Theorem \ref{thm:main thm} in Section \ref{sec:intrinsic}.

\subsection*{Acknowledgments}This work was initiated and partially carried out during the workshop ``Categorification in Low Dimensional Topology'' at Ruhr University Bochum, and we would like to thank the organizers for providing a stimulating environment. We would also like to thank Jianfeng Lin, Ciprian Manolescu, and Paul Wedrich for helpful discussions. QR was partially supported by the Simons Investigator Award 376200. HY was partially supported by the Simons Collaboration grant on New Structures in Low-Dimensional Topology and a Simons Dissertation Fellowship.

\section{Preliminaries}\label{sec:pre}

\subsection{Links in \texorpdfstring{$\RP^3$}{RP3}}\label{sec:links_RP3}
For our purposes, it is convenient to think of $\RP^3$ as the $3$-ball $B^3$ with antipodal points on the boundary sphere identified. We fix a distinguished point $\infty\in\RP^3$ corresponding to the (identified) north and south poles in the $B^3$ model, and let $\nu(\infty)\subset\RP^3$ denote a tubular neighborhood of it. We can then view $\RP^3\backslash\nu(\infty)$ as the solid cylinder with antipodal points identified on its side, or more precisely $$\RP^3\backslash\nu(\infty)=([-1,1]\times D^2)/((t,x)\sim(-t,-x)\colon t\in[-1,1],\ x\in S^1).$$
From this model, it is also clear that $\RP^3\backslash\nu(\infty)$ is the twisted $I$-bundle over $\RP^2$, denoted by $I\tilde\times\RP^2$, with the projection to $\RP^2=D^2/(x\sim-x\colon x\in S^1)$ given by $(t,z)\mapsto z,\,t\in[-1,1],\,z\in D^2$. The boundary $2$-sphere of $I\tilde\times\RP^2$ is shown as the upper and lower boundaries of $[-1,1]\times D^2$, identified along their common boundary circle.

Any link in $\RP^3$ can be perturbed to lie in $I\tilde\times\RP^2=\RP^3\backslash\nu(\infty)$, and the resulting link type is independent of the perturbation. In general position, a link in $I\tilde\times\RP^2$ has a generic link diagram on $\RP^2$ with only double point singularities. Since $\RP^2$ is nonorientable, one cannot distinguish overstrands from understrands at a crossing unless a local orientation at the crossing is chosen. Nevertheless, the notions of positive/negative crossings and $0/1$-resolutions at a crossing are still well-defined. In practice, it is sometimes convenient to draw link diagrams on $D^2$ with antipodal points on the boundary identified.

A link $L\subset\RP^3$ is of \textit{class-0} if it is null-homologous, and of \textit{class-1} if it is homologically essential. See Figure~\ref{fig:links_RP3} for an example of a two-component class-1 link in the cylinder model of $\RP^3\backslash\nu(\infty)$ and its diagram in $\RP^2$ drawn on the disk. For convenience, in our pictures, disks are slightly distorted to squares with round corners.

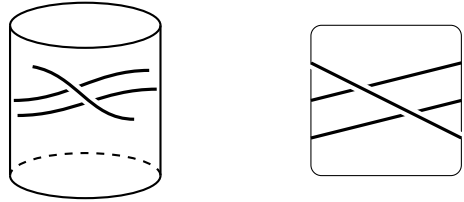
\begin{figure}[htbp]
    \centering
    \begin{tikzpicture}
    \node at (-2,0) {
    \begin{tikzpicture}[anchorbase,scale=0.5]
            \coordinate (O) at (0,0); 
    \coordinate (T) at (0,4); 
    \draw[thick] ($(O) + (-2,0)$) -- ($(T) + (-2,0)$); 
    \draw[thick] ($(O) + (2,0)$) -- ($(T) + (2,0)$);
    
    \draw[dashed, thick] ($(O) + (2,0)$) arc (0:180:2 and 0.6);
    \draw[thick] ($(O) - (2,0)$) arc (180:360:2 and 0.6);

    \draw[thick] ($(T) + (2,0)$) arc (0:180:2 and 0.6);
    \draw[thick] ($(T) - (2,0)$) arc (180:360:2 and 0.6) ;

    \draw[very thick] (-1.9,2) to[out=0,in=180] (1.7,2.8);
    \draw[very thick] (-1.8,1.6) to[out=0,in=180] (1.9,2.3);
    \draw[white,line width=0.15cm] (-1.4,2.9) to[out=-10,in=180] (1.3,1.5);
    \draw[very thick] (-1.4,2.9) to[out=-10,in=180] (1.3,1.5);
    \end{tikzpicture}
    };

    \node at (2,0) {
    \begin{tikzpicture}[anchorbase,scale=0.5]
        \draw[rounded corners] (0,0) rectangle (4,4);
        \draw[very thick] (0,1)--(4,2) (0,2)--(4,3);
        \draw[white,line width=0.15cm] (0,3)--(4,1);
        \draw[very thick] (0,3)--(4,1);
    \end{tikzpicture}
    };
\end{tikzpicture}
    \caption{A two-component class-1 link in $I\tilde\times\RP^2$ drawn in the cylinder model (left), and its diagram on $\RP^2$ drawn in the disk model (right). On the left, the ends of the strands are on the boundary of the cylinder.}
    \label{fig:links_RP3}
\end{figure}

\subsection{Khovanov homology in \texorpdfstring{$\RP^3$}{RP3}}\label{sec:Kh_RP3}
Asaeda, Przytycki, and Sikora \cite{asaeda2004categorification} define Khovanov homology for oriented links in oriented $I$-bundles over surfaces. As a special case, their construction gives Khovanov homology with $\FF_2$ coefficients for oriented links in $\RP^3$. Subsequently, sign fixes have been achieved by Gabrov\v sek \cite{Gabrovek2013THECO} and Manolescu--Willis \cite{manolescu2025rasmussen} to upgrade Khovanov homology to $\Z$ coefficients. In $\RP^3$, deformations of Khovanov homology in the spirit of Lee \cite{lee2005endomorphism} and Bar-Natan \cite{bar2005khovanov}*{Section~9.3} have also been studied \cites{manolescu2025rasmussen,chen2025bar}, and $s$-type filtration invariants in the spirit of Rasmussen \cite{rasmussen2010khovanov} have been extracted.

The construction of these Khovanov-type homologies with $\ZZ$ coefficients is best understood as a two-step process, as explained by Bar-Natan \cite{bar2005khovanov}: first, construct a link invariant in the Bar-Natan category of $\RP^2$; second, apply a TQFT-type functor from the Bar-Natan category to an abelian category to get homological invariants.

We briefly review the first step of the construction. Let $\LinkIP$ denote the category of oriented links in $I\tilde\times\RP^2$ with generic link diagrams on $\RP^2$, and oriented link cobordisms between them up to isotopy rel boundary. Let $\BNP$ denote the universal dotted Bar-Natan category of $\RP^2$, which is by definition the $\Z[E_1,E_2]$-linear $\Z$-graded category whose objects are embedded closed compact $1$-manifolds on $\RP^2$ with formal quantum $\Z$-degree shifts allowed, and whose morphisms between two objects $q^{m_1}O_1$ and $q^{m_2}O_2$ consist of $\Z[E_1,E_2]$-linear combinations of dotted cobordisms in $I\times\RP^2$ between $O_1$ and $O_2$ of degree $m_2-m_1$, up to isotopy rel boundary and Bar-Natan skein relations. Here, the degree of a dotted cobordism is its Euler characteristic minus twice the number of dots, and the degrees of $E_1,\,E_2$ are $-2,\,-4$, respectively. The Bar-Natan skein relations are given by
\begin{itemize}
\item An undotted local $2$-sphere evaluates to $0$;
\item A one-dotted local $2$-sphere evaluates to $1$;
\item Two adjacent dots can be replaced by $E_1$ times one dot minus $E_2$ times no dot, as shown in Figure~\ref{fig:double dotted};
\item The neck-cutting relation as shown in Figure~\ref{fig:neck_cutting}.
\end{itemize}
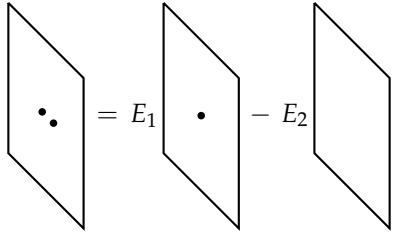
\begin{figure}[htbp]
    \centering
    \begin{equation*}
    \begin{tikzpicture}[anchorbase,scale=0.5]
        \draw[thick] (0,0)--(2,-2)--(2,-6)--(0,-4)--(0,0);
        \fill (1.2,-3.2) circle (0.1);
        \fill (0.9,-2.9) circle (0.1);
    \end{tikzpicture}
    \,=\,
    E_1\,\begin{tikzpicture}[anchorbase,scale=0.5]
        \draw[thick] (0,0)--(2,-2)--(2,-6)--(0,-4)--(0,0);
        \fill (1,-3) circle (0.1);
    \end{tikzpicture}
    \,-\,
    E_2\,\begin{tikzpicture}[anchorbase,scale=0.5]
        \draw[thick] (0,0)--(2,-2)--(2,-6)--(0,-4)--(0,0);
    \end{tikzpicture}
    \end{equation*}
    \caption{The dot relation in $\BNP$.}
    \label{fig:double dotted}
\end{figure}

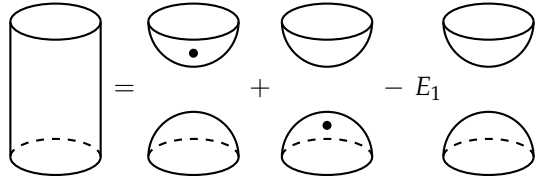
\begin{figure}[htbp]
    \centering
    \begin{equation*}
    \begin{tikzpicture}[anchorbase,scale=0.6]
        \twoelp
        \draw[thick] ($(O) + (-1,0)$) -- ($(T) + (-1,0)$); 
        \draw[thick] ($(O) + (1,0)$) -- ($(T) + (1,0)$);
    \end{tikzpicture}
    \,=\,
    \begin{tikzpicture}[anchorbase,scale=0.6]
        \twoelp
        \draw[thick] ($(O) + (1,0)$) arc (0:180:1 and 1);   
        \draw[thick] ($(T) - (1,0)$) arc (180:360:1 and 1);
        \fill (0,2.3) circle (0.1);
    \end{tikzpicture}
    \,+\,
    \begin{tikzpicture}[anchorbase,scale=0.6]
        \twoelp
        \draw[thick] ($(O) + (1,0)$) arc (0:180:1 and 1);   
        \draw[thick] ($(T) - (1,0)$) arc (180:360:1 and 1);
        \fill (0,0.7) circle (0.1);
    \end{tikzpicture}
    \,-\,E_1\begin{tikzpicture}[anchorbase,scale=0.6]
        \twoelp
        \draw[thick] ($(O) + (1,0)$) arc (0:180:1 and 1);   
        \draw[thick] ($(T) - (1,0)$) arc (180:360:1 and 1);
    \end{tikzpicture}
    \end{equation*}
    \caption{The neck-cutting relation in $\BNP$.}
    \label{fig:neck_cutting}
\end{figure}

Let $\Kbpm(\BNP)$ denote the bounded homotopy category of chain complexes over the additive category $\BNP$, with each morphism identified with its negative. At the level of Bar-Natan categories, functoriality of Khovanov homology up to overall sign in $I\tilde\times\RP^2$ can be phrased as having a functor
\begin{equation}\label{eq:BN_RP2}
\lrb{-}\colon\LinkIP\to\Kbpm(\BNP).
\end{equation}
We remark that both $\LinkIP$ and $\BNP$ decompose into two subcategories depending on the homology class of the objects. Thus, \eqref{eq:BN_RP2} can be phrased separately for class-0 and class-1 links. We will refer to the functor \eqref{eq:BN_RP2} as \textit{Bar-Natan's bracket invariant}.

The construction of \eqref{eq:BN_RP2} follows from the work of Bar-Natan \cite{bar2005khovanov} with little extra effort (see \cite{bar2005khovanov}*{Section~11.6}) as we now explain. At the level of objects in $\LinkIP$, since positive/negative crossings and $0/1$-resolutions are distinguished from each other at each crossing of a link diagram, one can define and flatten a cube of resolutions to obtain a chain complex over $\BNP$, with the appropriate homological and quantum gradings. At the level of morphisms in $\LinkIP$, one decomposes a generic link cobordism into Reidemeister moves, Morse moves, and planar isotopies of link diagrams. For each elementary move, one assigns an induced chain map as in \cite{bar2005khovanov}. Note that for Reidemeister moves, a local orientation on $\RP^2$ in the region where the move takes place has to be chosen before interpreting the assignment of the chain maps. Nevertheless, one can show that the two choices give equal assignments for Reidemeister I and II moves, and chain homotopic assignments for Reidemeister III moves (see the third and fourth paragraphs in \cite{morrison2022invariants}*{Section~3.5} for an effortless proof of the last claim). Finally, to show that two isotopic generic link cobordisms define chain homotopic chain maps, note that all the movie moves one has to check take place locally, so Bar-Natan's proof in \cite{bar2002khovanov}*{Section~8} remains valid.

To get module-valued invariants, the second step of the construction is to apply a functor from $\BNP$ to an abelian category and takes cohomology. For instance, this procedure recovers the functoriality for links in $\RP^3\backslash\nu(\infty)=I\tilde\times\RP^2$ and link cobordisms between them for
\begin{enumerate}[(a)]
\item The original Khovanov homology, with $\FF_2$ coefficients for class-0 links and $\Z$ coefficients for class-1 links \cite{asaeda2004categorification}*{Theorem~5.1(2)};
\item The Bar-Natan homology for class-0 links \cite{chen2025bar};
\item The Lee-type deformations of the Khovanov homology \cite{manolescu2025rasmussen}, with $\FF_2$ coefficients for class-0 links and $\Z$ coefficients (or $\QQ$ coefficients by specialization) for class-1 links.
\end{enumerate}

In this paper, we examine the Khovanov homology in $\RP^3$ at the level of Bar-Natan categories. It is easy to recover the corresponding statements for these module-valued invariants as special cases.

\begin{exmp}
Let $V=R[X]/(X^2-E_1X+E_2)$ be the commutative graded Frobenius algebra over the graded ring $R:=\Z[E_1,E_2]$, with multiplication $m$, unit $\eta$, comultiplication \[\Delta(1)=1\otimes X+X\otimes1-E_11\otimes1,\,\Delta(X)=X\otimes X-E_21\otimes1,\]and counit \[\epsilon(1)=0,\,\epsilon(X)=1.\] Let $\overline{V}$ be a (graded) $V$-module, namely a (graded) $R$-module equipped with a (graded) $V$-action $\overline{m}\colon\overline{V}\otimes V\to\overline{V}$ and a (graded) $V$-coaction $\overline{\Delta}\colon\overline{V}\to\overline{V}\otimes V$, subject to the relation $$(\overline{m}\otimes\idd_V)\circ(\idd_{\overline{V}}\otimes\Delta)=\overline{\Delta}\circ\overline{m}=(\idd_{\overline{V}}\otimes m)\circ(\overline{\Delta}\otimes\idd_V).$$
Let $\BNP[1]$ denote the full subcategory of $\BNP$ consisting of homologically essential objects. Then, there is a functor $F_{\overline{V}}$ from $\BNP[1]$ to the category of (graded) abelian groups, given at the level of objects by sending a collection of circles to some $\overline{V}\otimes V^{\otimes m}$, where the unique essential circle contributes to the factor $\overline{V}$, and each contractible circle contributes to a factor $V$. The assignment of $F_{\overline{V}}$ on morphisms is determined by the structural maps $\eta,\,\epsilon,\,m,\,\Delta,\,\overline{m}$, and $\overline{\Delta}$. The universal version of the homology theory (c) above for class-1 links is obtained by setting \[\overline{V}\coloneqq(R[s,t]/(E_1,E_2-st))\{1,\overline{X}\}\] with structural maps given by the second row of \cite{manolescu2025rasmussen}*{Table~1}.
\end{exmp}

\begin{rem}
The Khovanov homology in $\RP^3$ for class-0 links with $\Z$ coefficients requires intricate sign fixes as done in \cites{Gabrovek2013THECO,manolescu2025rasmussen}, and it does not seem readily recoverable from \eqref{eq:BN_RP2}. An alternative approach would be to use the $\mathfrak{gl}_2$ webs and foams formalism of \cite{blanchet2010oriented}. Once local orientations are chosen at each crossing of a link diagram in $\RP^2$, the cubes of resolution procedure yields a chain complex in the bounded homotopy category of the appropriate foam category. We expect that an appropriate functor from the foam category would then yield the $\Z$-coefficient Khovanov homology. However, at the level of foam categories themselves, the choices of local orientations affect cyclic orderings of the edges around trivalent vertices, and we do not know how to remove the dependence. We leave the proper setup along these lines for future work.
\end{rem}

\section{The sweep-around move}\label{sec:sw}
Let $\LinkP$ denote the category of oriented links in $\RP^3$ that are contained in $I\tilde\times\RP^2$ with generic link diagrams on $\RP^2$, and oriented link cobordisms between them up to isotopy rel boundary. There is a forgetful functor $$\LinkIP\to\LinkP$$ that is bijective on objects and full, but not faithful. The failure of faithfulness is due to the additional \textit{sweep-around move}, which we explain in Section~\ref{sec:move}. The goal of this section is to upgrade Bar-Natan's bracket invariant \eqref{eq:BN_RP2} to a functor from $\LinkP$, proving Theorem~\ref{thm:functoriality}.

Let $\LinkP[\epsilon]$ (resp. $\LinkIP[\epsilon]$) denote the full subcategory of $\LinkP$ (resp. $\LinkIP$) consisting of class-$\epsilon$ links, and $\BNP[\epsilon]$ denote the full subcategory of $\BNP$ consisting of objects with homology class $\epsilon[\RP^1]$, $\epsilon=0,1$.

\begin{thm}\label{thm:functoriality_1}
On the subcategory of class-1 links, \eqref{eq:BN_RP2} descends to a functor $$\lrb{-}\colon\LinkP[1]\to\Kbpm(\BNP[1]).$$
\end{thm}

Let $\BNP[0]'$ denote the quotient of $\BNP[0]$ with the following relations on morphisms imposed:
\begin{itemize}
\item $2=0$;
\item The undotted $\RP^2$ lying in a level of $I\times\RP^2$ evaluates to zero;
\item The one-dotted $\RP^2$ lying in a level of $I\times\RP^2$ evaluates to a square root of $E_1$.
\end{itemize}
In other words, $\BNP[0]'$ is the $\FF_2[E_1,E_2,P_\bullet]/(P_\bullet^2+E_1)$-linear category whose objects agree with $\BNP[0]$ and whose morphisms are linear combinations of cobordisms up to Bar-Natan skein relations and the following extra relations:
\begin{itemize}
\item The undotted $\RP^2$ lying in a level of $I\times\RP^2$ evaluates to zero;
\item The one-dotted $\RP^2$ lying in a level of $I\times\RP^2$ evaluates to $P_\bullet$.
\end{itemize}

\begin{thm}\label{thm:functoriality_0}
On the subcategory of class-0 links, \eqref{eq:BN_RP2} descends to a functor $$\Kh\colon\LinkP[0]\to\Kbpm(\BNP[0]').$$
\end{thm}

\begin{rem}
The category $\BNP[0]'$ is a further quotient of the universal target, i.e., the minimal quotient of $\BNP[0]$, that makes Theorem~\ref{thm:functoriality_0} hold. See Section~\ref{sec:proof_functoriality}.
\end{rem}

\begin{rem}\label{rem:concrete_functorality}
Theorem~\ref{thm:functoriality} applies to each of the theories (a)(b)(c) in Section~\ref{sec:Kh_RP3}.
\end{rem}

\subsection{The move}\label{sec:move}
To check \eqref{eq:BN_RP2} descends to $\LinkP$ (after passing to a quotient in the target in the class-0 case), it suffices to check that if $\Sigma,\Sigma'\colon L_0\to L_1$ are two generic link cobordisms in $I\times(I\tilde\times\RP^2)$ that are isotopic rel boundary in $I\times\RP^3$, then they induce the same chain map up to sign and homotopy on Bar-Natan's bracket invariants.

By a general position argument, one can assume an isotopy rel boundary in $I\times\RP^3$ between $\Sigma$ and $\Sigma'$ intersects the segment $I\times\{\infty\}$ transversely in finitely many points at distinct time slices. Away from these time slices, one can use the functoriality in $I\tilde\times\RP^2$ to interpolate between isotopies. Therefore, it suffices to prove the invariance of the induced maps across each of these time slices where the isotopy passes through infinity. By further decomposing and isotoping, this boils down to the case when $\Sigma\colon L\to L$ is the identity cobordism and $\Sigma'\colon L\to L$ is a ``sweep-around move'' (cf. \cite{morrison2022invariants}) that drags a strand of $L$ near $S^2=\partial(I\tilde\times\RP^2)$ to sweep it once around $S^2$. 

By further splitting an unknot off the moving strand, exploiting functoriality and noting that merging an unknot to $L$ induces a split surjection on $\Kbpm(\BNP)$, we reduce to the case when the moving strand sits on a split unknotted component of $L$. This is depicted in Figure~\ref{fig:sweeparound} in the cylinder model and Figure \ref{fig:sweeparound_diagram} at the diagram level shown in the disk model. Let $D_i$ denote the diagram drawn in the $i$-th frame in Figure \ref{fig:sweeparound_diagram}, $i=1,2,\dots,9$. Let $f_i$ denote the movie moves from $D_i$ to $D_{i+1}$ ($i=1,2,\dots,8$). Then $f_i$ is a planar isotopy on $\RP^2$ for $i=1,3,5$, $f_8$ is the flip map on the unknot component, and the remaining $f_i$'s can be realized by Reidemeister 2 and 3 moves.

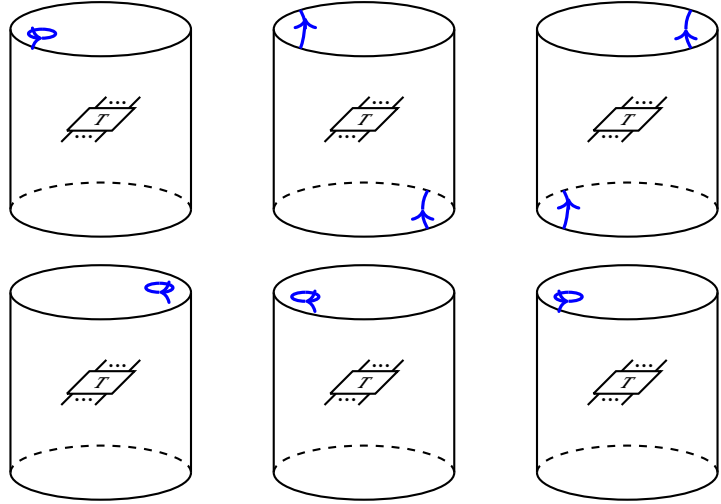
\begin{figure}[htbp]
    \centering
    \begin{tikzpicture}[anchorbase,scale=.7]
        \node at (-5,5) {\moveA{0.6}};
        \node at (0,5) {\moveB{0.6}};
        \node at (5,5) {\moveC{0.6}};
        \node at (-5,0) {\moveD{0.6}};
        \node at (0,0) {\moveE{0.6}};
        \node at (5,0) {\moveA{0.6}};
    \end{tikzpicture}
    \caption{The sweep-around move in $I\tilde\times\RP^2$ shown in the cylinder model. The first three maps isotope the unknot component near the whole boundary. The fourth map isotopes the unknot component near the upper boundary. The last map flips the unknot component locally.}
    \label{fig:sweeparound}
\end{figure}

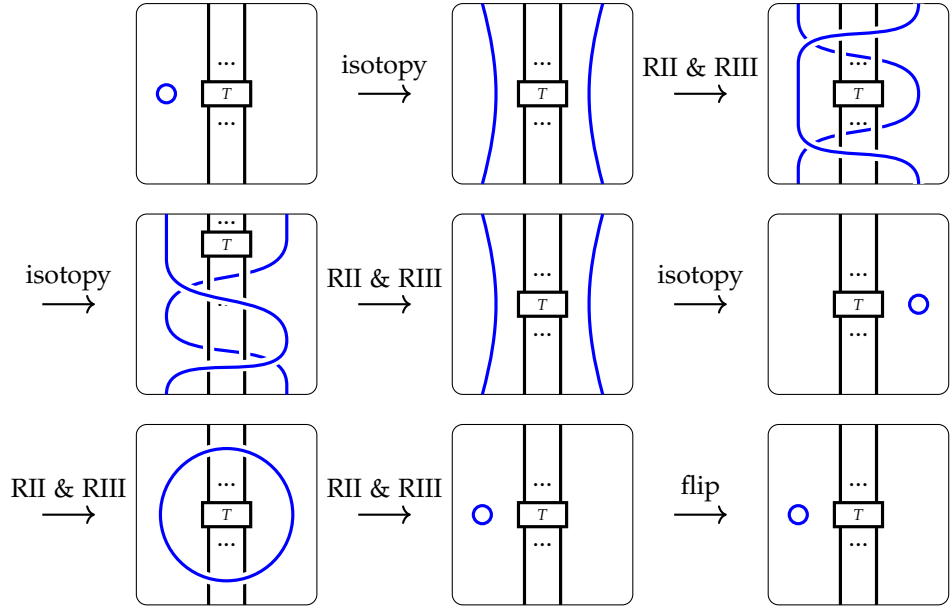
\begin{figure}[htbp]
    \centering
    \begin{tikzpicture}[anchorbase,scale=.7]

        \draw[->,thick] (-3.5,4)--(-2.5,4);
        \node at (-3,4.5) {\text{isotopy}};
        
        \draw[->,thick] (-3.5,0)--(-2.5,0);
        \node at (-3,0.5) {\text{RII \& RIII}};
        
        \draw[->,thick] (-3.5,-4)--(-2.5,-4);
        \node at (-3,-3.5) {\text{RII \& RIII}};

        \draw[->,thick] (2.5,4)--(3.5,4);
        \node at (3,4.5) {\text{RII \& RIII}};
        
        \draw[->,thick] (2.5,0)--(3.5,0);
        \node at (3,0.5) {\text{isotopy}};
        
        \draw[->,thick] (2.5,-4)--(3.5,-4);
        \node at (3,-3.5) {\text{flip}};

        \draw[->,thick] (-9.5,0)--(-8.5,0);
        \node at (-9,0.5) {\text{isotopy}};
        
        \draw[->,thick] (-9.5,-4)--(-8.5,-4);
        \node at (-9,-3.5) {\text{RII \& RIII}};

        \node at (-6,4) {\diagmovA{0.4}};
        \node at (0,4) {\diagmovB{0.4}};
        \node at (6,4) {\diagmovC{0.4}};
        \node at (-6,0) {\diagmovD{0.4}};
        \node at (0,0) {\diagmovB{0.4}};
        \node at (6,0) {\diagmovE{0.4}};
        \node at (-6,-4) {\diagmovF{0.4}};
        \node at (0,-4) {\diagmovA{0.4}};
        \node at (6,-4) {\diagmovA{0.4}};
    \end{tikzpicture}
    \caption{The sweep-around move in $I\tilde\times\RP^2$ in diagram shown in the disk model.}
    \label{fig:sweeparound_diagram}
\end{figure}

In summary, to prove Theorems \ref{thm:functoriality_1} and \ref{thm:functoriality_0}, it suffices to prove the following proposition.
\newpage
\begin{prop}\label{prop:sweeparound}
    Let $D$ be a link diagram on $\RP^2$ of a crossingless unknot and the projective closure of a $(k,k)$-tangle $T$, drawn as in the first frame of Figure \ref{fig:sweeparound_diagram}. Then the map \[\sw_{D}\colon\lrb{D}\to\lrb{D}\]induced by the movie depicted in Figure \ref{fig:sweeparound_diagram} is the identity morphism in $\Kbpm(\BNP[0]')$ or $\Kbpm(\BNP[1])$.
\end{prop}

\begin{rem}
In the case of $S^3$, after detaching a local unknot from the moving strand, one can immediately deduce the triviality of the sweep-around move from Gujral--Levine \cite{gujral2022khovanov}*{Theorem~1.2} that the induced map of disjoint links cobordisms between split links is independent of the linking between components. Nevertheless, the proof presented in \cite{morrison2022invariants} provides more information and applies in the more general setup of Khovanov--Rozansky $\mathfrak{gl}_N$ homology.
\end{rem}

\subsection{Reduction to crossingless diagrams}

We argue that we can reduce the proof of Proposition~\ref{prop:sweeparound} to the case that $T$ has no crossings. For notational ease, we first reduce to the case when $T$ is a braid.

\begin{lem}\label{lem:nowbraid}
    It suffices to prove Proposition~\ref{prop:sweeparound} in the case that $T$ is a braid.
\end{lem}

\begin{proof}
     Let $T$ be an arbitrary $(k,k)$-tangle. Choose a braid $\beta$ such that the projective closures of $T$ and $\beta$ represent isotopic links in $I\tilde\times\RP^2$. Let $D$ and $D'$ be the diagrams of $T\cup U$ and $\beta\cup U$ respectively. Fix an isomorphism $\varphi\colon\lrb{D}\to\lrb{D'}$ induced by a sequence of Reidemeister moves. By functoriality of Khovanov homology in $I\tilde\times\RP^2$, the following diagram commutes:\[
    \begin{tikzcd}[row sep = 10mm, column sep =  10mm]
        \lrb{D} \ar[r,"\sw_D"] \ar[d,"\varphi"] &  \lrb{D} \ar[d,"\varphi"]\\
        \lrb{D'} \ar[r,"\sw_{D'}"] & \lrb{D'}
    \end{tikzcd}.
    \]Therefore it suffices to prove that $\sw_\beta=\idd$ in $\Kbpm(\BNP[0]')$ or $\Kbpm(\BNP[1])$ for all braids $\beta$.
\end{proof}

Now assume $T=\beta$ is a braid with $l$ strands and $N$ crossings. The formal complex $\lrb{D}$ carries a $\{0,1\}^N$-grading coming from resolving the crossings either by the $0$-resolution or by the $1$-resolution, which we call the \textit{cubical grading}. The homological grading on $\lrb{D}$ is the sum of the cubical grading shifted by a constant. The map $\sw_D\colon\lrb{D}\to\lrb{D}$ preserves the homological grading since the sweep-around cobordism has Euler characteristic $0$.

Each of $D_i$'s carries the same cubical grading from resolving the crossings in $\beta$, $i=1,2,\dots,9$. We claim that one can choose maps $\sw_i\colon\lrb{D_i}\to\lrb{D_{i+1}}$ induced by the movie moves $f_i$, $i=1,2,\dots,8$, such that each of them is filtered with respect to the cubical grading, and the composition of them is a representative of $\sw_D$. 

For $i=1,3,5$, let $\sw_i$ be the maps given by planar isotopies on $\RP^2$, and they trivially preserve the cubical grading. Let $\sw_4$ be the map induced by any sequence of Reidemeister moves happening away from $\beta$, and $\sw_8$ be the map induced by two Reidemeister 1 moves on the unknot component. They also preserve the cubical grading as they do not interact with $\beta$. 

Before choosing $\sw_i$ for $i=2,6,7$, we recall a key observation in \cite{morrison2022invariants}*{Section~3.5}. For $T_1$, $T_2$ the (local) source and target of a Reidemeister 3 move, there is a distinguished $1$-dimensional affine family of chain homotopy equivalences between $\lrb{T_1}$ and $\lrb{T_2}$, all chain homotopy equivalent to each other, that supplies the induced map for the Reidemeister 3 move \cite{elias2010rouquier}. By carefully choosing the parameter, one can arrange the chain map from $\lrb{T_1}$ to $\lrb{T_2}$ to be filtered with respect to the $\{0,1\}$-grading from resolving one preferred crossing. Moreover, the filter pieces of this map are induced by the natural sequence of Reidemeister moves.

In our case, we always pick the preferred crossings to be the ones in $\beta$. This is enough to choose $\sw_i$ for our purposes. For instance, $\sw_2$ can be realized by first performing $l$ Reidemeister 2 moves and $N$ Reidemeister 3 moves that drag the right blue component to the left (over $\beta$), then performing a Reidemeister 2 move between the two blue strands, finally performing $l$ Reidemeister 2 moves and $N$ Reidemeister 3 moves that drag the left blue component to the right (under $\beta$). All the Reidemeister 2 moves involved here do not affect the cubical grading, and by carefully choosing the maps realizing the Reidemeister 3 moves, we can ensure that $\sw_2$ is filtered with respect to the cubical grading. Similar arguments work for $\sw_6$ and $\sw_7$.

The composition of the $\sw_i$'s represents $\sw_D$. Before passing to homotopy, this specific representative is filtered with respect to the cubical grading and preserves the homological grading. Therefore, it must \textit{preserve} the cubical grading. In other words, it maps each vertex in the cube of resolutions to the corresponding vertex. This justifies the following lemma.

\begin{lem}\label{lem:nowcrossingless}
    It suffices to prove Proposition~\ref{prop:sweeparound} in the case that $T$ is a crossingless diagram.
\end{lem}

\begin{proof}
    Assume that we have proved that the sweep-around move induces the identity map in $\Kbpm(\BNP[0]')$ or $\Kbpm(\BNP[1])$ for crossingless diagrams. Let $D$ be an arbitrary diagram with $N$ crossings. As we argued above, at each cubical grading $v\in\{0,1\}^N$, $\sw_D$ is given by $\sw_{D_v}$, where $D_v$ is the resolved crossingless diagram at $v$. By assumption, $\sw_{D_v}=\pm\idd$ as chain maps, because $\lrb{D_v}$ is supported in a single homological grading. For two adjacent vertices $v$ and $v'$, the signs for $\sw_{D_v}$ and $\sw_{D_{v'}}$, as maps induced by these specific sequences of Reidemeister moves (and hence without sign ambiguity), must be the same because the merge or split cobordism between $D_v$ and $D_{v'}$ is nontrivial. Therefore, the sweep-around move for $D$ must also induce the chain map $\pm\idd$.
\end{proof}

\subsection{Model computations for crossingless diagrams}

We now state some local ingredients to compute the sweep-around map for crossingless diagrams. As usual in the Bar-Natan category, the statements are interpreted up to overall signs.

The first lemma below describes the effect of moving a strand passing over/under a crossingless unknot. The second one describes the flip map on a crossingless unknot.

\begin{figure}[htbp]
    \centering
    \begin{tikzpicture}[anchorbase,scale=.7]

    \draw[->,thick] (-2.5,0)--(-1.5,0);
    \draw[->,thick] (1.5,0)--(2.5,0);
    \node at (-2,0.5) {\text{RII}};
    \node at (2,0.5) {\text{RII}};
    
    \node at (-4,0) {
        \begin{tikzpicture}[anchorbase,scale=0.25]
            \draw[very thick] (0,0) circle (1);
            \draw[very thick] (2,-2)--(2,2);
        \end{tikzpicture}};
    \node at (0,0) {
        \begin{tikzpicture}[anchorbase,scale=0.25]
            \draw[very thick] (0,-2)--(0,2);
            \draw[white,line width=0.15cm] (0,0) circle (1);
            \draw[very thick] (0,0) circle (1);
        \end{tikzpicture}
    };
    \node at (4,0) {
        \begin{tikzpicture}[anchorbase,scale=0.25]
            \draw[very thick] (0,0) circle (1);
            \draw[very thick] (-2,-2)--(-2,2);
        \end{tikzpicture}
    };
        
    \end{tikzpicture}
    \caption{A sequence of Reidemeister moves that realizes a strand passing through a crossingless unknot.}
    \label{fig:strandmoving}
\end{figure}
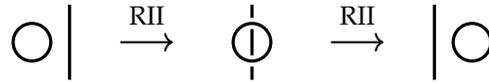

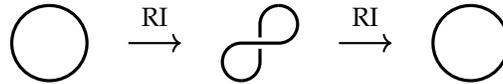
\begin{figure}[htbp]
    \centering
    \begin{tikzpicture}[anchorbase,scale=.7]

    \draw[->,thick] (-2.5,0)--(-1.5,0);
    \draw[->,thick] (1.5,0)--(2.5,0);
    \node at (-2,0.5) {\text{RI}};
    \node at (2,0.5) {\text{RI}};
    
    \node at (-4,0) {
        \begin{tikzpicture}[anchorbase,scale=0.25]
            \draw[very thick] (0,0) circle (2);
        \end{tikzpicture}};
    \node at (0,0) {
        \begin{tikzpicture}[anchorbase,scale=0.25]
            \draw[very thick]  (-1,-2) to[out=0,in=-90] (0,-1)--(0,1) to[out=90,in=180] (1,2);
            \draw[white,line width=0.15cm] (-1,-2) to[out=180,in=-90] (-2,-1) to[out=90,in=180] (-1,0)--(1,0) to[out=0,in=-90] (2,1) to[out=90,in=0] (1,2);
            \draw[very thick] (-1,-2) to[out=180,in=-90] (-2,-1) to[out=90,in=180] (-1,0)--(1,0) to[out=0,in=-90] (2,1) to[out=90,in=0] (1,2);
        \end{tikzpicture}
    };
    \node at (4,0) {
        \begin{tikzpicture}[anchorbase,scale=0.25]
            \draw[very thick] (0,0) circle (2);
        \end{tikzpicture}
    };
    \end{tikzpicture}
    \caption{A sequence of Reidemeister moves that realizes the flip cobordism for the crossingless unknot.}
    \label{fig:flipunknot}
\end{figure}

\begin{lem}\label{lem:strandmoving}
    The induced map of the movie depicted in Figure~\ref{fig:strandmoving} is given by 
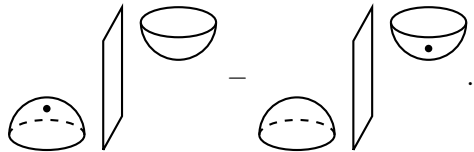
\begin{equation*}
    \begin{tikzpicture}[anchorbase,scale=0.5]
        \draw[dashed, thick] (1,0) arc (0:180:1 and 0.4);
        \draw[thick] (-1,0) arc (180:360:1 and 0.4);
        \draw[thick] (1,0) arc (0:180:1 and 1); 
        \fill (0,0.7) circle (0.1);

        \draw[thick] (1.5,-0.4)--(2,0.5)--(2,3.4)--(1.5,2.5)--(1.5,-0.4);
        
        \draw[thick] (4.5,3) arc (0:180:1 and 0.4);
        \draw[thick] (2.5,3) arc (180:360:1 and 0.4) ;
        \draw[thick] (2.5,3) arc (180:360:1 and 1);
    \end{tikzpicture}
    \,-\,
    \begin{tikzpicture}[anchorbase,scale=0.5]
        \draw[dashed, thick] (1,0) arc (0:180:1 and 0.4);
        \draw[thick] (-1,0) arc (180:360:1 and 0.4);
        \draw[thick] (1,0) arc (0:180:1 and 1); 

        \draw[thick] (1.5,-0.4)--(2,0.5)--(2,3.4)--(1.5,2.5)--(1.5,-0.4);
        
        \draw[thick] (4.5,3) arc (0:180:1 and 0.4);
        \draw[thick] (2.5,3) arc (180:360:1 and 0.4) ;
        \draw[thick] (2.5,3) arc (180:360:1 and 1);
        \fill (3.5,2.3) circle (0.1);
    \end{tikzpicture}.
    \end{equation*}The same result holds if the strand passes under the unknot.\qed
\end{lem}

\begin{lem}\label{lem:flipunknot}
   The induced map of the movie depicted in Figure~\ref{fig:flipunknot} is given by 
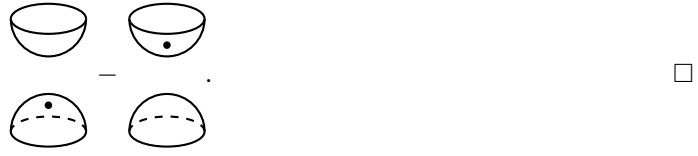
\begin{equation*}
    \begin{tikzpicture}[anchorbase,scale=0.5]
        \draw[dashed, thick] (1,0) arc (0:180:1 and 0.4);
        \draw[thick] (-1,0) arc (180:360:1 and 0.4);
        \draw[thick] (1,0) arc (0:180:1 and 1); 
        \fill (0,0.7) circle (0.1);
        \draw[thick] (1,3) arc (0:180:1 and 0.4);
        \draw[thick] (-1,3) arc (180:360:1 and 0.4) ;
        \draw[thick] (-1,3) arc (180:360:1 and 1);
    \end{tikzpicture}
    \,-\,
    \begin{tikzpicture}[anchorbase,scale=0.5]
        \draw[dashed, thick] (1,0) arc (0:180:1 and 0.4);
        \draw[thick] (-1,0) arc (180:360:1 and 0.4);
        \draw[thick] (1,0) arc (0:180:1 and 1); 
        \draw[thick] (1,3) arc (0:180:1 and 0.4);
        \draw[thick] (-1,3) arc (180:360:1 and 0.4) ;
        \draw[thick] (-1,3) arc (180:360:1 and 1);
        \fill (0,2.3) circle (0.1);
    \end{tikzpicture}.\tag*{\qed}
    \end{equation*}
\end{lem}

Note that the effects of the move on the unknot components are the same for the moves in Figure \ref{fig:strandmoving} and Figure \ref{fig:flipunknot}.

The next two ingredients are stated locally in the M\"obius band $M$, thought of as the tubular neighborhood of the segment $\{0\}\times I$ in the disk model of $\RP^2=D^2/\sim$. These can be compared with annular link homology theories, cf. \cites{grigsby2018annular,queffelec2018annular}.

\begin{figure}[htbp]
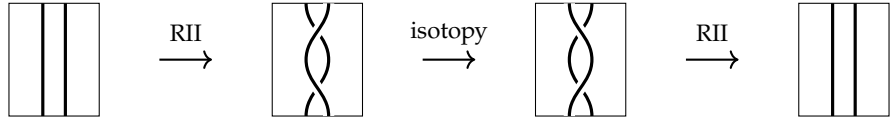

    \centering
    \classomiddlemove
    \caption{Reidemeister moves involved from the second frame to the fifth frame in Figure \ref{fig:sweeparound_diagram} when $T$ is empty. }
    \label{fig:class0middlemove}
\end{figure}

\begin{lem}\label{lem:class0middlemove}
    The induced map of the movie depicted in Figure \ref{fig:class0middlemove} is given by 
    
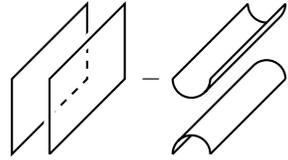
\begin{equation*}
    \begin{tikzpicture}[anchorbase,scale=0.25]
        \draw[dashed, thick] (2,2)--(4,4)--(4,6);
        \draw[thick] (2,2)--(0,0)--(0,4)--(4,8)--(4,6);
        \draw[white, line width=0.15cm] (2,0)--(6,4)--(6,8)--(2,4)--(2,0);
        \draw[thick] (2,0)--(6,4)--(6,8)--(2,4)--(2,0);
        
    \end{tikzpicture}
    \,-\,
    \begin{tikzpicture}[anchorbase,scale=0.25]

        \draw[thick] (0,4)--(4,8) (2,4)--(6,8)
                     (0,4) arc (180:360:1 and 1)
                     (4,8) arc (180:270:1 and 1)
                     ($(5,8)+(-45:1)$) arc (315:360:1 and 1)
                     ($(1,4)+(-45:1)$)--($(5,8)+(-45:1)$);
        \draw[thick] (0,0)--(1,1) (2,0)--(6,4)
                     (2,0) arc (0:180:1 and 1) 
                     (6,4) arc (0:135:1 and 1)
                     ($(1,0)+(135:1)$)--($(5,4)+(135:1)$)
                     ; 
    \end{tikzpicture}.
    \end{equation*}
\end{lem}

\begin{proof}
    This is also a routine calculation. Enumerate the crossings in the second frame from bottom to top. The nonzero terms appear for (10) and (01) resolutions. The former gives the first summand in the induced map, and the latter gives the second summand. Compare \cite{grigsby2018annular}*{Figure~10}.
\end{proof}

\begin{lem}\label{lem:class1middlemove}
    The induced map of the movie depicted in Figure \ref{fig:class1middlemove} is the identity.
\end{lem}

\begin{figure}[htbp]
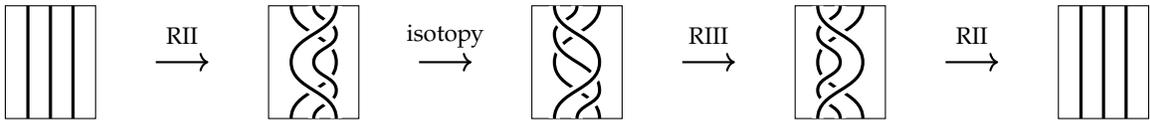

    \centering
    \classimiddlemove
    \caption{Reidemeister moves involved from the second frame to the fifth frame in Figure \ref{fig:sweeparound_diagram} when $T$ is a single strand.}
    \label{fig:class1middlemove}
\end{figure}

\begin{proof}
One can check this explicitly by working out the induced maps of all the Reidemeister moves involved. We provide a concise topological proof here. Let $C\cup C_\partial\subset M$ denote the domain and target depicted in Figure~\ref{fig:class1middlemove}, where $C\subset M$ is the central circle, and $C_\partial\subset M$ is a $\partial$-parallel circle. We claim that $C\cup C_\partial$ is a simple object in the Bar-Natan category $\mathbf{BN}_M$, meaning that the only degree $0$ endomorphisms from $C\cup C_\partial$ to itself are integral scalars of the identity morphism. Since the induced map by Figure~\ref{fig:class1middlemove} is an invertible endomorphism of $C\cup C_\partial$ of degree $0$, this would imply the lemma.

It suffices to show that the only essential surface in $I\times M$ bounding $\partial I\times(C\cup C_\partial)$ with nonnegative Euler characteristic is the product surface $I\times(C\cup C_\partial)$. Let $\Sigma$ be such an essential surface. Note that $I\times M$ is the nonorientable $3$-manifold with one $0$-handle and one $1$-handle. By cutting along the cocore disk of the $1$-handle and applying the innermost disk argument, we see that $I\times M$ is irreducible, contains no $\RP^2$, no incompressible $T^2$ or Klein bottles.
Since moreover $\{j\}\times C,\{j\}\times C_\partial$ represent homology classes $1,2\in\Z\cong H_1(I\times M)$, respectively, $j=0,1$, we see that each component of $\Sigma$ has Euler characteristic $0$ and nonempty boundary, and that there is one component $\Sigma_1$ of $\Sigma$ that cobounds $\partial I\times C$. By cutting along the cocore of the $1$-handle (chosen as $I\times\gamma$ for a properly embedded arc $\gamma$ transversely intersecting $C$ at one point) and applying the innermost disk argument again, we see that $\Sigma_1=I\times C$, and consequently $\Sigma\backslash\Sigma_1=I\times C_\partial$ by looking at the complement of $\Sigma_1$ in $I\times M$.
\end{proof}

\subsection{Completion of the proofs}\label{sec:proof_functoriality}

Diagrams survived from Lemma~\ref{lem:nowcrossingless} are all crossingless, but they can still look bizarre. Nevertheless, by a functoriality argument in the spirit of Lemma~\ref{lem:nowbraid}, one can assume the local unknot components of $T$ are all local on the nose, i.e., they do not intersect to the drawn boundary $\RP^1$ of $\RP^2$. In this case, we can then use the model computations in the last subsection to deduce the triviality of the sweep-around move. Again, we need to discuss in terms of the homology class of the link.

\begin{proof}[Proof of Theorem~\ref{thm:functoriality_1}]

By Lemma~\ref{lem:nowcrossingless}, we may assume $T$ is a crossingless diagram. By functoriality, we may further assume that one component of $T$ is the axis $x=0$, $y\in[0,1]$ in the disk model of $\RP^2$, and all other components are local unknots contained in the interior of $D^2$, each of which passes once under and once over the moving circle throughout the entire sweep-around move. 

In the case when $T$ has only one component, the sweep-around move as shown in Figure~\ref{fig:sweeparound_diagram} is decomposed into
\begin{itemize}
\item From the first frame to the second frame: a planar isotopy;
\item From the second frame to the fifth frame: the moves shown in Figure~\ref{fig:class1middlemove};
\item From the fifth frame to the sixth frame: a planar isotopy;
\item From the sixth frame to the eighth frame: the moves shown in Figure~\ref{fig:strandmoving};
\item From the eighth frame to the last frame: a flip, as shown in Figure~\ref{fig:flipunknot}.
\end{itemize}
By Lemmas~\ref{lem:strandmoving}, \ref{lem:flipunknot}, and \ref{lem:class1middlemove}, noting the relation 
\begin{equation*}
    \left(\begin{tikzpicture}[anchorbase,scale=0.5]
        \draw[dashed, thick] (1,0) arc (0:180:1 and 0.4);
        \draw[thick] (-1,0) arc (180:360:1 and 0.4);
        \draw[thick] (1,0) arc (0:180:1 and 1); 
        \fill (0,0.7) circle (0.1);
        \draw[thick] (1,3) arc (0:180:1 and 0.4);
        \draw[thick] (-1,3) arc (180:360:1 and 0.4) ;
        \draw[thick] (-1,3) arc (180:360:1 and 1);
    \end{tikzpicture}
    \,-\,
    \begin{tikzpicture}[anchorbase,scale=0.5]
        \draw[dashed, thick] (1,0) arc (0:180:1 and 0.4);
        \draw[thick] (-1,0) arc (180:360:1 and 0.4);
        \draw[thick] (1,0) arc (0:180:1 and 1); 
        \draw[thick] (1,3) arc (0:180:1 and 0.4);
        \draw[thick] (-1,3) arc (180:360:1 and 0.4) ;
        \draw[thick] (-1,3) arc (180:360:1 and 1);
        \fill (0,2.3) circle (0.1);
    \end{tikzpicture}\right)
    \circ\left(
    \begin{tikzpicture}[anchorbase,scale=0.5]
        \draw[dashed, thick] (1,0) arc (0:180:1 and 0.4);
        \draw[thick] (-1,0) arc (180:360:1 and 0.4);
        \draw[thick] (1,0) arc (0:180:1 and 1); 
        \fill (0,0.7) circle (0.1);
        \draw[thick] (1,3) arc (0:180:1 and 0.4);
        \draw[thick] (-1,3) arc (180:360:1 and 0.4) ;
        \draw[thick] (-1,3) arc (180:360:1 and 1);
    \end{tikzpicture}
    \,-\,
    \begin{tikzpicture}[anchorbase,scale=0.5]
        \draw[dashed, thick] (1,0) arc (0:180:1 and 0.4);
        \draw[thick] (-1,0) arc (180:360:1 and 0.4);
        \draw[thick] (1,0) arc (0:180:1 and 1); 
        \draw[thick] (1,3) arc (0:180:1 and 0.4);
        \draw[thick] (-1,3) arc (180:360:1 and 0.4) ;
        \draw[thick] (-1,3) arc (180:360:1 and 1);
        \fill (0,2.3) circle (0.1);
    \end{tikzpicture}\right)=\idd,
\end{equation*}
we conclude that the sweep-around move induces the identity morphism in $\Kbpm(\BNP[1])$. In general, when $T$ has $k$ local unknot components, by extra $2k$ applications of Lemma~\ref{lem:strandmoving}, we reach the same conclusion.

\end{proof}

\begin{proof}[Proof of Theorem~\ref{thm:functoriality_0}]

As in the previous case, by Lemma~\ref{lem:nowcrossingless}, we may assume that $T$ consists of local unknots. In the case when $T=\emptyset$, the whole move shown in Figure~\ref{fig:sweeparound_diagram} is decomposed into
\begin{itemize}
\item From the first frame to the second frame: a planar isotopy;
\item From the second frame to the fifth frame: the move shown in Figure~\ref{fig:class0middlemove};
\item From the fifth frame to the eighth frame: a planar isotopy;
\item From the eighth frame to the last frame: a flip, as shown in Figure~\ref{fig:flipunknot}.
\end{itemize}
By Lemmas~\ref{lem:strandmoving}, \ref{lem:flipunknot}, and \ref{lem:class0middlemove}, the composition induces the map \begin{equation}\label{eq:class0composition}
    \left(\begin{tikzpicture}[anchorbase,scale=0.5]
        \draw[dashed, thick] (1,0) arc (0:180:1 and 0.4);
        \draw[thick] (-1,0) arc (180:360:1 and 0.4);
        \draw[thick] (1,0) arc (0:180:1 and 1); 
        \fill (0,0.7) circle (0.1);
        \draw[thick] (1,3) arc (0:180:1 and 0.4);
        \draw[thick] (-1,3) arc (180:360:1 and 0.4) ;
        \draw[thick] (-1,3) arc (180:360:1 and 1);
    \end{tikzpicture}
    \,-\,
    \begin{tikzpicture}[anchorbase,scale=0.5]
        \draw[dashed, thick] (1,0) arc (0:180:1 and 0.4);
        \draw[thick] (-1,0) arc (180:360:1 and 0.4);
        \draw[thick] (1,0) arc (0:180:1 and 1); 
        \draw[thick] (1,3) arc (0:180:1 and 0.4);
        \draw[thick] (-1,3) arc (180:360:1 and 0.4) ;
        \draw[thick] (-1,3) arc (180:360:1 and 1);
        \fill (0,2.3) circle (0.1);
    \end{tikzpicture}\right)\circ\left(\begin{tikzpicture}[anchorbase,scale=0.5]
        \draw[dashed, thick] (1,0) arc (0:180:1 and 0.4);
        \draw[thick] (-1,0) arc (180:360:1 and 0.4);
        \draw[thick] (1,3) arc (0:180:1 and 0.4);
        \draw[thick] (-1,3) arc (180:360:1 and 0.4) ;
        \draw[thick] (-1,0)--(-1,3) (1,0)--(1,3);
    \end{tikzpicture}
    \,-\,
    \begin{tikzpicture}[anchorbase,scale=0.5]
        \draw[dashed, thick] (1,0) arc (0:180:1 and 0.4);
        \draw[thick] (-1,0) arc (180:360:1 and 0.4);
        \draw[thick] (1,0) arc (0:180:1 and 1); 
        \draw[thick] (1,3) arc (0:180:1 and 0.4);
        \draw[thick] (-1,3) arc (180:360:1 and 0.4) ;
        \draw[thick] (-1,3) arc (180:360:1 and 1);
        \node[font=\small] at (0,2.3) {\text{P}};
        \node[font=\small] at (0,0.7) {\text{P}};
    \end{tikzpicture}\right),
\end{equation}
where each of the $P$-labeled disk shown in the last term denotes the complement in $\RP^2$ of the disk bounded by the moving unknot at its initial position, slightly pushed into the $I$-direction, which is topologically the connected sum of a disk and $\RP^2$.

The composition \eqref{eq:class0composition}, as a morphism from the unknot $U\subset\RP^2$ to itself, is not equal to the identity morphism or its negative in $\BNP[0]$, thus we need to pass to some suitable quotient of $\BNP[0]$.

Let $R$ denote the non-commutative ring $\mathrm{End}_{\BNP[0]}(\emptyset)$ and $P,P_\bullet\in R$ denote the elements represented by the undotted, dotted $\{1/2\}\times\RP^2$ in $I\times\RP^2$, respectively. By capping off the top and bottom of the composition by undotted/dotted disks, we find that the following relations must hold in the corresponding quotient of $R$:
\begin{equation}\label{eq:class0fundamentalequations}
\begin{cases}
P^2=0,\\
PP_\bullet=1\pm1,\\
P_\bullet P=1\mp1,\\
P_\bullet^2=E_1.
\end{cases}
\end{equation}
Conversely, if \eqref{eq:class0fundamentalequations} holds, by exploiting the delooping isomorphism $U\cong q\emptyset\oplus q^{-1}\emptyset$ in $\BNP[0]$, we conclude that \eqref{eq:class0composition} is equal to $\pm1$ in the corresponding quotient of $\BNP[0]$. By some further applications of Lemma~\ref{lem:strandmoving} and reverse neck-cuttings, the same conclusion applies when $T$ is nonempty. Thus, the set of equations \eqref{eq:class0fundamentalequations} provides the minimal set of relations on $R=\mathbf{End}_{\BNP[0]}(\emptyset)$ ensuring that the sweep-around move induces the identity map in the corresponding bounded homotopy category, up to sign. The category $\BNP[0]'$ appearing in Theorem~\ref{thm:functoriality_0} is a convenient further quotient of this minimal requirement by demanding that $R$ has no nontrivial zero divisors.

\end{proof}

\section{Intrinsic invariant}\label{sec:intrinsic}
An \textit{unparametrized} (oriented) $\RP^3$ is an oriented $3$-manifold diffeomorphic to $\RP^3$. We prove the following precise versions of Theorem~\ref{thm:main thm} at the level of Bar-Natan categories, in the spirit of \cite{morrison2022invariants}*{Section~4}.

\begin{thm}\label{thm:unpara_Kh_1}
Let $M$ be an unparametrized $\RP^3$ and $L\subset M$ be a class-1 oriented link. There is a chain complex $\lrb{(M,L)}$ over $\BNP[1]$ whose chain homotopy type is an invariant of $(M,L)$, functorial up to sign under diffeomorphisms of the pair. When $M=\RP^3$, this recovers Bar-Natan's bracket invariant of $L$.
\end{thm}

Let $\BNP[0]''$ be the quotient of $\BNP[0]'$ by setting $E_1=0$.

\begin{thm}\label{thm:unpara_Kh_0}
Let $M$ be an unparametrized $\RP^3$ and $L\subset M$ be a class-0 oriented link. There is a chain complex $\lrb{(M,L)}$ over $\BNP[0]''$ whose chain homotopy type is an invariant of $(M,L)$, functorial up to sign under diffeomorphisms of the pair. When $M=\RP^3$, this recovers Bar-Natan's bracket invariant of $L$.
\end{thm}

\begin{rem}
In the case of $S^3$, \cite{morrison2022invariants} further shows that link cobordisms in unparametrized $I\times S^3$ induce well-defined chain maps up to homotopy between the invariants of the source and target links. We have not yet established a corresponding result in our case, which would require more understanding of the smooth mapping class group of $I\times\RP^3$ rel boundary up to $\pi_0(\Diff(S^4))$.
\end{rem}

\begin{rem}\label{rem:concrete_unpara}
    Theorem~\ref{thm:main thm} applies to each of the theories (a)(c) in Section~\ref{sec:Kh_RP3}.
\end{rem}

\begin{rem}
    It would also be interesting to see if Theorem~\ref{thm:functoriality_0} and Theorem~\ref{thm:unpara_Kh_0} apply (possibly with modifications) to Chen's Khovanov-type homologies for class-0 links in $\RP^3$ \cite{chen2021khovanov}, which are related to the Heegaard Floer homology of branched double covers of $\RP^3$ by a spectral sequence.
\end{rem}

Let $(M,L)$ be a pair as in Theorem~\ref{thm:unpara_Kh_1} or Theorem~\ref{thm:unpara_Kh_0}. An (orientation-preserving) parametrization $\phi\colon M\xrightarrow{\cong}\RP^3$ is \textit{$L$-admissible} if $\phi(L)$ is an object of $\LinkP$. We would like to define $\lrb{(M,L)}$ as $\lrb{\phi(L)}$ for any $L$-admissible parametrization $\phi$ of $M$. In order for this to be well-defined, we have to assign, to any two $L$-admissible parametrizations $\phi_0,\phi_1$ of $M$, a chain map $\lrb{\phi_0(L)}\to\lrb{\phi_1(L)}$ well-defined up to sign and homotopy, in a functorial way. By \cite{bamler2019ricci}, \[\Diff(M,\RP^3)=\{f\colon M\to\RP^3\colon f\text{ is an orientation-preserving diffeomorphism}\}\]has the same homotopy type as the isometry group, $\SO(4)/\{\pm1\}$, of $\RP^3$ with the elliptic metric. In particular, it is connected, and we have \[\pi_1(\Diff(\RP^3))\cong\ZZ\oplus\ZZ.\]We can then choose a path $\phi_t$, $t\in I$, of parametrizations of $M$ between $\phi_0$ and $\phi_1$. Now $\phi_t(L)$ gives an isotopy between $\phi_0(L)$ and $\phi_1(L)$, thus by Theorem~\ref{thm:functoriality_1} or Theorem~\ref{thm:functoriality_0} it induces a chain map 
\begin{equation}\label{eq:identification}
\lrb{(\phi_t(L))_{t\in I}}\colon\lrb{\phi_0(L)}\to\lrb{\phi_1(L)}
\end{equation}
well-defined up to sign and homotopy, which a priori depends only on the homotopy class of the path $(\phi_t)_{t\in I}$ in $\Diff(M,\RP^3)$ rel endpoints. Once we remove the dependence on the homotopy class of $(\phi_t)_{t\in I}$ as well, the map \eqref{eq:identification} would supply a canonical chain map between $\lrb{\phi_0(t)}$ and $\lrb{\phi_1(t)}$ up to sign and homotopy which is functorial as one can concatenate paths of isotopies.

To remove the dependence on the homotopy class of $\phi_t$, it suffices to show that a concatenation with any generator of $\pi_1(\Diff(\RP^3))$ does not change the induced map. Note that $\Diff(\RP^3)$ fits into the fibration $$\Diffb(\RP^3)\hookrightarrow\Diff(\RP^3)\to\RP^3,$$ where the second map is given by evaluating at the base point $\infty\in\RP^3$, and the fiber is the based diffeomorphism group. By taking the homotopy long exact sequence, one realizes two generators of $\pi_1(\Diff(\RP^3))$ represented by
\begin{enumerate}[(i)]
\item The $2\pi$-rotation along the $z$-axis in the ball model of $\RP^3=B^3/\sim$. The effect of this isotopy in the presence of a link is shown at the level of diagrams in Figure~\ref{fig:2pi_rot}.
\begin{figure}
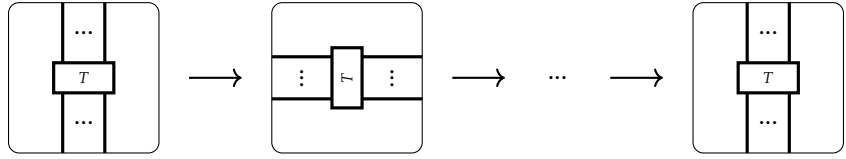

    \centering
    \typeimonodromy
    \caption{The effect of the type (i) isotopy on a link diagram.}
    \label{fig:2pi_rot}
\end{figure}
\item A full translation along the $y$-axis segment $\{0\}\times[-1,1]\times\{0\}$ in the ball model of $\RP^3$, composed with a $\pi$-rotation along the $y$-axis segment. Up to isotopy, the effect of this isotopy in the presence of a projective braid closure is shown at the level of diagrams in Figure~\ref{fig:translation_flip}.
\begin{figure}[htbp]
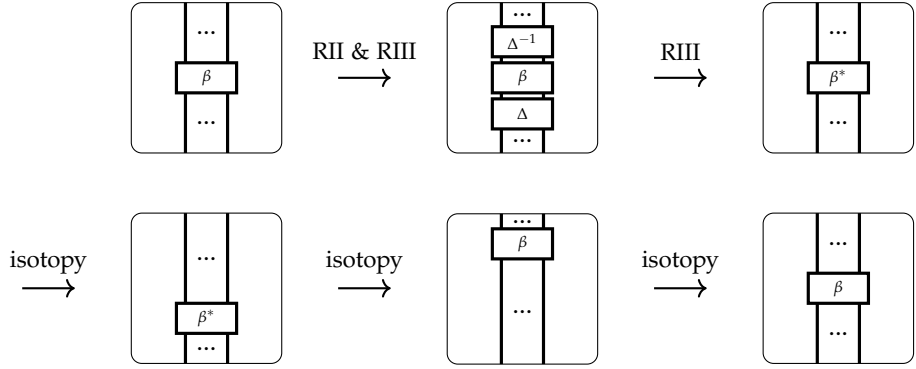

    \centering
    \typeiimonodromy
    \caption{The effect of the type (ii) isotopy on a link diagram. Here $\beta^*$ denotes the horizontal flip of the braid $\beta$, and $\Delta$ denotes a positive half twist.}
    \label{fig:translation_flip}
\end{figure}
\end{enumerate}

A type (i) isotopy isotopes the diagram of a link in $\LinkIP$ by planar rotations. We argue that it always induces the identity map in $\Kbpm(\BNP)$. It suffices to prove this for each crossingless diagram. By functoriality, we may take the crossingless diagram to be some small concentric circles around the origin in the disk model of $\RP^2$, with an additional essential circle $S^1/\sim$ in the case of class-1 link diagrams. It is then clear that a type (i) isotopy induces the identity map.

In the rest of this section, we analyze the effect on the Bar-Natan's bracket invariant of a type (ii) isotopy of a projective braid closure, as in Figure~\ref{fig:translation_flip}. To prove Theorems~\ref{thm:unpara_Kh_1} and \ref{thm:unpara_Kh_0}, it suffices to prove the following proposition in $\Kbpm(\BNP[0]')$ or $\Kbpm(\BNP[1])$.

\begin{prop}\label{prop:type_ii_idd}
    Let $L\subset \RP^3$ be a class-1 (resp. class-0) oriented link, presented as a projective braid closure $D$ of a braid $\beta$. Then the map \[\fl\colon\lrb{D}\to\lrb{D}\]induced by the movie depicted in Figure~\ref{fig:translation_flip} is the identity map in $\Kbpm(\BNP[1])$ (resp. $\Kbpm(\BNP[0]'')$).
\end{prop}

The movie depicted in Figure~\ref{fig:translation_flip} is essentially the flip cobordism studied in \cite{chen2025flip}. To write it down clearly in terms of Reidemeister moves, we proceed as in \cite{chen2025flip}*{Proposition~4.4}. Namely, we introduce pairs of half twists (by Reidemeister 2 moves) between crossings in $\beta$ and flip each crossing separately using the fact that $\Delta^{-1}\sigma^{\pm 1}\Delta\cong(\sigma^{\pm 1})^*$, where $\sigma$ is a generator of the braid group. Let $N$ denote the number of crossings in $\beta$. The formal complex $\lrb{D}$ carries a $\{0,1\}^N$, or cubical, grading from the cube of resolutions.

\begin{lem}\label{lem:filtered_type_ii}
    Under the setup of Proposition \ref{prop:type_ii_idd}, there is a chain map $\widetilde{\fl}\colon\lrb{D}\to\lrb{D}$ representing $\fl$ in the homotopy category that is filtered with respect to the cubical grading on $\lrb{D}$. Moreover, the restriction of $\widetilde{\fl}$ on each summand $\lrb{D_v}$ of $\lrb{D}$, where $v$ is a vertex in the cube of resolutions, is given by the same flip map on $D_v$, i.e., performing the movie depicted in Figure~\ref{fig:translation_flip} on the resolution $D_v$.
\end{lem}

\begin{proof}
    Argument similar to \cite{chen2025flip}*{Lemma 4.5} applies here with some slight modifications. First, \cite{chen2025flip}*{Lemma 4.5} is stated using Khovanov's arc algebra formalism \cite{khovanov2002functor} over $\FF_2$, but the argument remains valid without change under Bar-Natan's formalism over $\Z$ (up to overall sign), as remarked in \cite{chen2025flip}*{Remark 4.6}. Second, the ordering in the half twists changes when an isotopy passing through the boundary of the disk happens (for instance, from the second frame in Figure \ref{fig:class1middlemove} to the third), but this is irrelevant to the cubical grading.
\end{proof}

We can then use the map $\widetilde{\fl}$ in Lemma~\ref{lem:filtered_type_ii} to compute $\fl$. This map is filtered with respect to the cubical grading and preserves the homological grading, so $\widetilde{\fl}$ must preserve the cubical grading. As in Lemma \ref{lem:nowcrossingless}, we again reduce the proof of Proposition~\ref{prop:type_ii_idd} to the case that $D$ is a crossingless diagram. The calculations for class-0 and class-1 crossingless diagrams are still slightly different, and we treat them separately.

\begin{proof}[Proof of Proposition~\ref{prop:type_ii_idd}]

\begin{figure}[htbp]
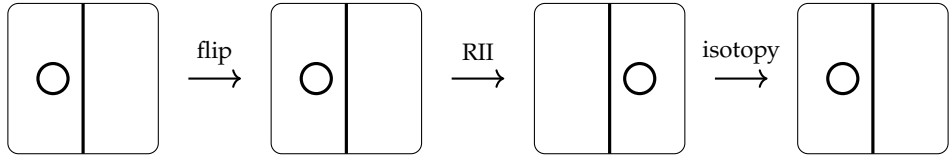

    \centering
    \classitypeiimove
    \caption{A type (ii) isotopy on a crossingless class-1 diagram.}
    \label{fig:class1typeiimove}
\end{figure}

    We first assume $D$ is a class-1 link diagram. By Lemma~\ref{lem:filtered_type_ii}, we may assume $D$ is crossingless. By functoriality, we may further assume the homologically essential component in $D$ is the $y$-axis circle $\{0\}\times[-1,1]/\sim$ in the disk model of $\RP^2$, and all other components are local unknots. Under the type (ii) isotopy, the homologically essential component in $D$ stays the same, and each unknot component first flips, passes over to the other side of the axis, and then travels back by a planar isotopy, as depicted in Figure~\ref{fig:class1typeiimove}. The moves from the first frame to the second and from the second to the third are the ones shown in Figure~\ref{fig:flipunknot} and Figure~\ref{fig:strandmoving} respectively. By Lemmas~\ref{lem:strandmoving} and \ref{lem:flipunknot}, we conclude that the type (ii) isotopy induces the identity morphism in $\Kbpm(\BNP[1])$ since \begin{equation*}
    \left(\begin{tikzpicture}[anchorbase,scale=0.5]
        \draw[dashed, thick] (1,0) arc (0:180:1 and 0.4);
        \draw[thick] (-1,0) arc (180:360:1 and 0.4);
        \draw[thick] (1,0) arc (0:180:1 and 1); 
        \fill (0,0.7) circle (0.1);
        \draw[thick] (1,3) arc (0:180:1 and 0.4);
        \draw[thick] (-1,3) arc (180:360:1 and 0.4) ;
        \draw[thick] (-1,3) arc (180:360:1 and 1);
    \end{tikzpicture}
    \,-\,
    \begin{tikzpicture}[anchorbase,scale=0.5]
        \draw[dashed, thick] (1,0) arc (0:180:1 and 0.4);
        \draw[thick] (-1,0) arc (180:360:1 and 0.4);
        \draw[thick] (1,0) arc (0:180:1 and 1); 
        \draw[thick] (1,3) arc (0:180:1 and 0.4);
        \draw[thick] (-1,3) arc (180:360:1 and 0.4) ;
        \draw[thick] (-1,3) arc (180:360:1 and 1);
        \fill (0,2.3) circle (0.1);
    \end{tikzpicture}\right)
    \circ\left(
    \begin{tikzpicture}[anchorbase,scale=0.5]
        \draw[dashed, thick] (1,0) arc (0:180:1 and 0.4);
        \draw[thick] (-1,0) arc (180:360:1 and 0.4);
        \draw[thick] (1,0) arc (0:180:1 and 1); 
        \fill (0,0.7) circle (0.1);
        \draw[thick] (1,3) arc (0:180:1 and 0.4);
        \draw[thick] (-1,3) arc (180:360:1 and 0.4) ;
        \draw[thick] (-1,3) arc (180:360:1 and 1);
    \end{tikzpicture}
    \,-\,
    \begin{tikzpicture}[anchorbase,scale=0.5]
        \draw[dashed, thick] (1,0) arc (0:180:1 and 0.4);
        \draw[thick] (-1,0) arc (180:360:1 and 0.4);
        \draw[thick] (1,0) arc (0:180:1 and 1); 
        \draw[thick] (1,3) arc (0:180:1 and 0.4);
        \draw[thick] (-1,3) arc (180:360:1 and 0.4) ;
        \draw[thick] (-1,3) arc (180:360:1 and 1);
        \fill (0,2.3) circle (0.1);
    \end{tikzpicture}\right)=\idd.
\end{equation*}

Now let $D$ be a class-0 link diagram. Again, we may assume $D$ is a crossingless diagram by Lemma~\ref{lem:filtered_type_ii}. By functoriality, we may further assume that all components in $D$ are local unknots. The effect of a type (ii) isotopy on $D$ is simply to flip each unknot component and then isotope it back. By Lemma~\ref{lem:flipunknot}, it remains to prove \begin{equation}\label{eqn:flip=id}
    \begin{tikzpicture}[anchorbase,scale=0.5]
        \draw[dashed, thick] (1,0) arc (0:180:1 and 0.4);
        \draw[thick] (-1,0) arc (180:360:1 and 0.4);
        \draw[thick] (1,0) arc (0:180:1 and 1); 
        \fill (0,0.7) circle (0.1);
        \draw[thick] (1,3) arc (0:180:1 and 0.4);
        \draw[thick] (-1,3) arc (180:360:1 and 0.4) ;
        \draw[thick] (-1,3) arc (180:360:1 and 1);
    \end{tikzpicture}
    \,-\,
    \begin{tikzpicture}[anchorbase,scale=0.5]
        \draw[dashed, thick] (1,0) arc (0:180:1 and 0.4);
        \draw[thick] (-1,0) arc (180:360:1 and 0.4);
        \draw[thick] (1,0) arc (0:180:1 and 1); 
        \draw[thick] (1,3) arc (0:180:1 and 0.4);
        \draw[thick] (-1,3) arc (180:360:1 and 0.4) ;
        \draw[thick] (-1,3) arc (180:360:1 and 1);
        \fill (0,2.3) circle (0.1);
    \end{tikzpicture}=\idd.
\end{equation}In the quotient category $\BNP[0]''$, we have $2=0$ and $E_1=0$, so the neck-cutting relation implies (\ref{eqn:flip=id}), which completes the proof.
\end{proof}

\bibliographystyle{amsalpha}
\bibliography{ref}

\end{document}